\documentclass[a4paper,11 pt,twoside]{amsart}

\usepackage[utf8]{inputenc}
\usepackage[english]{babel}
\usepackage{times}
\usepackage{amsmath}
\usepackage{amsfonts}
\usepackage{amssymb}
\usepackage{amsmath}
\usepackage{amsthm}
\usepackage{enumerate}
\usepackage{hyperref}
\usepackage{fancyhdr}
\usepackage{mathrsfs} 
\usepackage{tikz}

\usepackage{stmaryrd}
\renewcommand{\dots}{\ifmmode\mathinner{\ldotp\kern-0.2em\ldotp\kern-0.2em\ldotp}\else.\kern-0.13em.\kern-0.13em.\fi}

\newtheorem{theorem}{Theorem}
\newtheorem{lemma}{Lemma}
\newtheorem{proposition}[lemma]{Proposition}

\newtheorem{remark}[lemma]{Remark}

\definecolor{darkgreen}{rgb}{0.1,0.7,0.1}

\definecolor{darkred}{rgb}{0.7,0.1,0.1}
\definecolor{darkblue}{rgb}{0.1,0.1,0.7}

\newcommand{\E}{\mathbb{E}}
\renewcommand{\P}{\mathbb{P}}

\newcommand{\bP}{\mathbf{P}}

\newcommand{\bX}{\mathbf{X}}

\newcommand{\bbE}{\mathbb{E}}

\newcommand{\bbN}{\mathbb{N}}

\newcommand{\bbP}{\mathbb{P}}

\newcommand{\bbR}{\mathbb{R}}

\newcommand{\bbZ}{\mathbb{Z}}

\newcommand{\cA}{\mathcal{A}}
\newcommand{\cB}{\mathcal{B}}
\newcommand{\cC}{\mathcal{C}}

\newcommand{\cF}{\mathcal{F}}
\newcommand{\cG}{\mathcal{G}}

\newcommand{\cL}{\mathcal{L}}

\newcommand{\cP}{\mathcal{P}}

\newcommand{\cR}{\mathcal{R}}
\newcommand{\cS}{\mathcal{S}}
\newcommand{\cT}{\mathcal{T}}

\newcommand{\Beta}{\mathrm{Beta}}

\newcommand{\ga}{\alpha}

\newcommand{\gd}{\delta}
\newcommand{\gep}{\varepsilon}       

\newcommand{\gG}{\Gamma}

\newcommand{\gO}{\Omega}
\newcommand{\gl}{\lambda}

\newcommand{\ind}{\mathbf{1}}
\DeclareMathOperator{\gap}{\mathrm{gap}}
\newcommand{\lint}{\llbracket}
\newcommand{\rint}{\rrbracket}

\DeclareMathSymbol{\leqslant}{\mathalpha}{AMSa}{"36} 
\DeclareMathSymbol{\geqslant}{\mathalpha}{AMSa}{"3E} 
\DeclareMathSymbol{\eset}{\mathalpha}{AMSb}{"3F}     
\renewcommand{\leq}{\;\leqslant\;}                   
\renewcommand{\geq}{\;\geqslant\;}                   
\newcommand{\dd}{\,\text{\rm d}}             

\newcommand{\inftwo}[2]{\inf_{\substack{#1 \\ #2}}} 

\newcommand{\var}{{\rm Var}}

\newcommand{\cc}{\complement}
\newcommand{\eq}{\mbox{\tiny eq}}
\renewcommand{\tilde}{\widetilde}

\renewcommand{\geq}{\ge}
\renewcommand{\leq}{\le}

\begin{document}

\title{Mixing time of the adjacent walk on the simplex}

\author{Pietro Caputo}
\address{Department of Mathematics and Physics, Roma Tre University, Largo San Murialdo 1, 00146 Roma, Italy.}
\email{caputo@mat.uniroma3.it}
\author{Cyril Labb\'e}
\address{Universit\'e Paris-Dauphine, PSL University, Ceremade, CNRS, 75775 Paris Cedex 16, France.}
\email{labbe@ceremade.dauphine.fr}
\author{Hubert Lacoin}
\address{IMPA, Estrada Dona Castorina 110, Rio de Janeiro, Brasil.}
\email{lacoin@impa.br}

\pagestyle{fancy}
\fancyhead[LO]{}
\fancyhead[CO]{\sc{P.~Caputo, C.~Labb\'e and H.~Lacoin}}
\fancyhead[RO]{}
\fancyhead[LE]{}
\fancyhead[CE]{\sc{Mixing time of the adjacent walk on the simplex}}
\fancyhead[RE]{}

\date{\small{January 26, 2020}}

\begin{abstract}
By viewing the $N$-simplex as the set of positions of $N-1$ ordered particles on the unit interval, the adjacent walk  is the continuous time Markov chain obtained by updating independently at rate 1 the position of each particle with a sample from the uniform distribution over the interval given by the two particles adjacent to it. We determine its spectral gap and prove that both the total variation distance and the separation distance to the uniform distribution exhibit a cutoff phenomenon, with mixing times that differ by a factor $2$. The results are extended to the family of log-concave distributions obtained by replacing the uniform sampling by a symmetric log-concave Beta distribution. 

\medskip

\noindent
{\bf MSC 2010 subject classifications}: Primary 60J25; Secondary 37A25, 82C22.\\
 \noindent
{\bf Keywords}: {\it Spectral gap; Mixing time; Cutoff; Adjacent walk.}
\end{abstract}

\maketitle

\setcounter{tocdepth}{1}
\tableofcontents

\section{Introduction}
Randomized algorithms based on Markov chains are commonly used for sampling points uniformly at random in a convex body and to simulate other log-concave distributions in $N$-dimensional Euclidean space. The mixing time of the associated random walk in $\bbR^N$ is often known to be polynomial in $N$, see e.g.\ \cite{dyer1991random,lovasz2003hit}.  
In analogy with the  setting of Markov chains with discrete state space, where the theory seems to be more advanced, see e.g.\ the monographs \cite{aldous2002reversible,LevPerWil}, it is of interest  
to develop a finer analysis of the asymptotic growth  of mixing times in high dimensions.
%

In this paper we address that question for 
a specific model, namely the adjacent walk on the $N$-simplex. 
By viewing the $N$-simplex as the set of positions of $N-1$ ordered particles on the unit interval, this is the process obtained by updating independently at rate 1 the position of each particle with a sample from the uniform distribution over the interval given by the two particles adjacent to it. The process defines a Gibbs sampler for the uniform distribution over the simplex. 

The adjacent walk on the $N$-simplex has been previously analysed in \cite{RW05}, where the authors
proved upper and lower bounds on the mixing time that are tight up to constant factors; see also \cite{RW05b,Smith1,Smith2} for similar estimates in related models. A version of this model with open boundary conditions was introduced in~\cite{KMP82} to study the heat flow in a chain of one-dimensional oscillators.  

Here we determine the mixing time to leading order and we establish the so-called cutoff phenomenon for
the adjacent walk on the $N$-simplex.  We also show that the same results hold if the uniform distribution over the sampling interval is replaced by a symmetric log-concave Beta law, in which case the chain converges to a log-concave distribution over the simplex.

\subsection{Model and results}
It is convenient to replace the unit interval by the interval $[0,N]$, that is to rescale the $N$-simplex to the set $\gO_N$ of positions of $N-1$ ordered particles on the interval $[0,N]$ defined by
\begin{equation}\label{Eq:omegan}
 \gO_N:=\left\{ x=(x_1,\dots,x_{N-1})\in \bbR^{N-1}: 0\le x_1 \le \dots \le x_{N-1} \le N\right\}.
\end{equation}
We also set $x_0=0,x_N=N$. An element $x$ in $\gO_N$ can be viewed as a non-decreasing interface $k \mapsto x_k$ connecting $(0,0)$ to $(N,N)$: this viewpoint will be adopted throughout the article.
Given an initial condition $x\in \gO_N$, we consider the process 
$${\bf X}^x(t)=(X_1(t),\dots,X_{N-1}(t)),$$ 
where ${\bf X}(0)=x$, and independently for each $i$, 
at rate $1$, $X_i$ is resampled according to the uniform distribution on $[X_{i-1},X_{i+1}]$, where we use the convention  $X_0 := 0$ and $X_N := N$.

More formally, this is the continuous time Markov chain with generator  
\begin{equation}\label{Eq:Generator}
(\cL_N f)(x) = \sum_{i=1}^{N-1} \int_{0}^1 \left(f(x^{(i,u)}) - f(x)\right)du\;,
\end{equation}
where, given $u\in [0,1]$ and $i$, the transformation $x\mapsto x^{(i,u)}$ is defined by
$$x^{(i,u)}_j =\begin{cases}x_j \;,\quad & i\ne j \;,\\
u\, x_{i-1} + (1-u) x_{i+1}\;,\quad & i=j\,.\end{cases}$$
 
Let $\pi_N$ denote the uniform distribution over $\gO_N$. Noting that $\int_{0}^1 f(x^{(i,u)}) du$ coincides with the conditional expectation $\pi_N[f\,|\,x_j, j\neq i]$ shows that the generator is a finite linear combination  of orthogonal projections in $L^2(\Omega_N,\pi_N)$. In particular, $\cL_N$ is a bounded self adjoint operator, and the Markov chain is reversible 
with respect to $\pi_N$. 
Note that $\pi_N$ can be obtained by conditioning $N$ i.i.d.~exponential random variables of parameter $1$ to have a sum equal to $N$, and by taking the $x_i$'s to be their partial sums.

The generator is nonpositive definite and has the trivial eigenvalue zero associated to constant functions. Our first result is a characterization of the spectral gap, defined as the smallest nontrivial eigenvalue of $-\cL_N$.

\begin{proposition}\label{gapz}
For any $N\geq 2$, the spectral gap of the generator is given by 
\begin{equation}\label{deff1}
\gap_N = 1 - \cos\left(\frac{\pi}{N}\right),
\end{equation} 
and the
corresponding eigenfunction is 
\begin{equation}\label{deffn1}
f_N (x) = \sum_{k=1}^{N-1} \sin\left( \frac{\pi k}{N}\right) (x_k-k)\;.
\end{equation}
\end{proposition}

We are interested in the time needed for the total variation distance to equilibrium, starting from the ``worst'' initial condition, to pass below some given threshold $\gep \in (0,1)$. When this time is, to leading order in $N$, insensitive to the choice of $\gep$, one speaks of a cutoff phenomenon. More precisely, we set
\begin{equation}\label{defdn} 
d_N(t) := \sup_{x\in\Omega_N} \| P_t^x - \pi_N \|_{TV}\;,
\end{equation}
where $P_t^x$ is the law of ${\bf X}^x(t)$, and for  $\mu,\nu$ probability measures on $\gO_N$, 
the total variation distance is given by 
$$
\|\mu-\nu\|_{TV}= \sup_{B\in \cB(\gO_N)} \left(\mu(B)-\nu(B)\right),
$$
the supremum ranging over all Borel subsets of $\gO_N$. 
For any $\gep \in (0,1)$, the mixing time is defined by
$$ T_N(\gep) := \inf\{t\ge 0: d_N(t) < \gep\}\;.$$
Our main result is the following.
\begin{theorem}\label{Th:Main}
For any $\gep \in (0,1)$,
$$ \lim_{N\to\infty} \frac{T_N(\gep)}{N^2 \log N} = \frac1{\pi^2}\;.$$
As a consequence, the sequence of Markov chains displays a cutoff phenomenon.
\end{theorem}

In view of Proposition \ref{gapz}, the above theorem can be restated as
$$  T_N(\gep) \stackrel{N\to \infty}{\sim} \frac{\log N}{2 \gap_N}.$$
Remark that 
if $f_N$ is as in \eqref{deffn1} and we start with the extremal initial condition $x_i\equiv N$, then $t=\frac{\log N}{2 \gap_N}$ is exactly the time it takes for the expected value 
$$
\bbE\left[f_N({\bf X}^x(t))\right]=f_N(x)\,e^{-\gap_Nt}\,,
$$
to drop from the initial value $f_N(x) = \Theta(N^2)$ to a value $\Theta(N^{3/2})$, which is the typical size of fluctuations of $f_N$ at equilibrium.  

Strikingly, Proposition \ref{gapz} and Theorem \ref{Th:Main} take the exact same form in the case of the symmetric simple exclusion process on the $N$-segment; see \cite{Wil04,Lac16}. 

\subsection{Generalization to Beta-resampling}
Given $\alpha\in (0,\infty)$,  the symmetric Beta distribution of parameter 
$\alpha$ is the probability measure on $[0,1]$ with density
\begin{equation}\label{Eq:Beta} 
\rho_{\alpha}(u):=\frac{\gG(2\alpha)}{\gG(\alpha)^2} [u(1-u)]^{\alpha-1}.
\end{equation}
We define a generalization of the process described in \eqref{Eq:Generator} by resampling points according to a Beta distribution
\begin{equation}\label{Eq:Generator2}
(\cL_{N,\alpha} f)(x) = \sum_{i=1}^{N-1} \int_{0}^1 \left(f(x^{(i,u)}) - f(x)\right)\rho_{\alpha}(u)du
\;.
\end{equation}
While $\rho_{\alpha}(u)du$ could be replaced by any probability on $[0,1]$, the   Beta laws are  the only resampling laws that make the dynamics reversible with respect to probability measures with a product structure; see Remark \ref{rem:betarev} below. The associated equilibrium distribution is given by 
\begin{equation}\label{Eq:pinalpha}
 \pi_{N,\alpha}( \dd x):=\frac{\gG(N\alpha)}{\gG(\alpha)^N N^{N\alpha-1}}\prod_{i=1}^N (x_i-x_{i-1})^{\alpha-1} \dd x
\end{equation}
where $\dd x$ is Lebesgue's measure on $\gO_N$. Since $\int_{0}^1 f(x^{(i,u)}) \rho_{\alpha}(u)du$ equals the conditional expectation $\pi_{N,\alpha}[f\,|\,x_j, j\neq i]$, it follows that $\cL_{N,\alpha}$ is a bounded self adjoint operator in $L^2(\Omega_N,\pi_{N,\alpha})$. The following theorem extends the results of Proposition \ref{gapz} and Theorem \ref{Th:Main}.

\begin{theorem}\label{Th:Maingen}
For any $\ga\geq 1$, the spectral gap and the corresponding eigenfunction are given by
$$\gap_{N} = 1 - \cos\left(\frac{\pi}{N}\right)\;,\quad f_N (x) = \sum_{k=1}^{N-1} \sin\left( \frac{\pi k}{N}\right) (x_k-k)\;.$$
Moreover, for any $\gep \in (0,1)$ the $\gep$-mixing time associated with the Beta-resampling process satisfies
$$ \lim_{N\to\infty} \frac{T_{N,\alpha}(\gep)}{N^2 \log N} = \frac1{\pi^2}\;.$$
In particular, the sequence of Markov chains displays a cutoff phenomenon.
\end{theorem}

\begin{remark}
 We have chosen $x_N:= N$ so that average inter-particle spacing at equilibrium is one. However, this convention has no influence on the result and we may take as well $x_N=1$ or any other constant (the only effect of this change being a dilation or contraction of space). In the course of the proof it is sometimes convenient to consider also the case of a random $x_N$ (since the process leaves $x_N$ fixed this makes the dynamics non-irreducible in that case).
\end{remark}

\begin{remark}
The restriction $\alpha\ge 1$ is due to the fact that certain parts of our proof require log-concavity of the probability density $\rho_{\alpha}$, which in turn implies the validity of the FKG property for the equilibrium measure $\pi_{N,\alpha}$, see Section \ref{sec:fkg}. However, we believe the result to be valid for all $\ga>0$.
\end{remark}

\begin{remark}
The spectrum of $\cL_{N,\alpha}$ restricted to the invariant subspace consisting of linear functions can be computed explicitly (see Section \ref{Sec:Eigen} below) and the fact that $\gap_{N,\alpha} = 1 - \cos\left(\frac{\pi}{N}\right)$ is equivalent to the statement that the spectral gap is attained within this subspace. 
As the following mean field example shows, this phenomenon is not  always to be expected for the exchange dynamics obtained by replacing the segment with another graph. Consider for instance the exchange process on the complete graph with generator 
\begin{equation}\label{Eq:GeneratorMF}
(\cG_{N,\alpha} f)(\eta) = \frac1N\sum_{1\leq i< j\leq N} \int_{0}^1 \left(f(\eta^{(i,j,u)})-f(\eta)\right) \rho_{\alpha}(u)du \,,
\end{equation}
where 
$$
\eta^{(i,j,u)}_k =\begin{cases}\eta_k \;,\quad & k\ne i,j \;,\\
u (\eta_{i} + \eta_{j})\;,\quad & k=i\;,\\
(1-u)(\eta_{i} + \eta_{j})\;,\quad & k=j\,,\end{cases}
$$
and  $\eta=\{\eta_k\}$ denotes the increment variables $\eta_k=x_k-x_{k-1}$, $k=1,\dots,N$. In other words, $\cG_{N,\alpha}$ defines the mean field version of \eqref{Eq:Generator2}. Notice that $\cG_{N,\alpha}$ is reversible w.r.t.\ $\pi_{N,\alpha}$. 
Using the arguments in \cite{CCL2003} or \cite{Cap08} one can show that for each $N\geq 2,\alpha>0$ the spectral gap of $ \cG_{N,\alpha}$ satisfies
\begin{equation}\label{Eq:GapMF8}
\gap(\cG_{N,\alpha})= 
\frac{\ga N +1}{(2\ga+1)N}\,,
\end{equation}
and that the corresponding eigenfunction is the symmetric quadratic function $$g_N(\eta) = {\rm const.}+\sum_{i=1}^N \eta_i^2.$$ 
\end{remark}

\subsection{Mixing time for the separation distance}

Another distance for which mixing time can be considered is the separation distance (see \cite{Gibbs02}). Its use in the study of mixing times for Markovian systems was initially advocated by Aldous and Diaconis 
because of its appealing relation to strong-stationary times \cite{AD87}.
While to our knowledge, separation distance has been considered so far only for countable state spaces, 
it can be generalized to a continuous setup via the definition
\begin{equation}\label{defseparation}
 d^{\,\mathrm{sep}}_{N,\alpha}(t):= 1- 
 \inftwo{A \in \cB(\gO_N)}{ B\in \cB(\gO_N) }\frac{ P_t(A,B)}{\pi_{N,\alpha}(A)\pi_{N,\alpha}(B)}
\end{equation}
where $$  P_t(A,B):= \int_{A} P^x_t(B)  \pi_{N,\alpha}(\dd x), $$
and
the infimum ranges over all Borel subsets of $\gO_N$ of positive Lebesgue measure. We can then define the separation mixing time associated with our chain as 
$$T^{\,\mathrm{sep}}_{N,\alpha}(\gep):= \inf\{ t\ge 0 \ : \  d^{\,\mathrm{sep}}_{N,\alpha}(t)<\gep\}.$$

\medskip

The separation mixing time is always larger than or equal to its total variation counterpart. It is known (cf. \cite[Lemma 19.3]{LevPerWil} in the discrete setup, and Section \ref{sec:sepa} below for an adaptation of the argument to the continuous case) that in the case of reversible chains, the separation mixing time is at most twice as large as that in total variation. Because of this factor $2$, cutoff in total variation does not necessarily imply that the same phenomenon holds for the separation distance (see \cite{HLP16} for counter-examples)

\medskip

We prove nonetheless that our dynamics displays a cutoff phenomenon for the separation distance and that the associated mixing times are twice as large as those for the total variation distance. 

\begin{theorem}\label{Th:Separation}
For all $\alpha\ge 1$ and every $\gep \in (0,1)$, the $\gep$-mixing time for the separation distance satisfies
$$ \lim_{N\to\infty} \frac{T^{\,\mathrm{sep}}_{N,\alpha}(\gep)}{N^2 \log N} = \frac{2}{\pi^2}\;.$$
\end{theorem}

Here again, our result is analogous to the one obtained for the symmetric simple exclusion process~\cite{Wil04,Lac16}.

\begin{remark}
Note that another natural definition of the separation distance in the continuous setup would be  
 \begin{equation}\label{defseparation2}
 d'_{N,\alpha}(t):= 1- 
 \inftwo{x\in \gO_N}{ B\in \cB(\gO_N) }\frac{ P_t(x,B)}{\pi_{N,\alpha}(B)}.
\end{equation}
Using explicit bounds on the total-mass of the singular part of the law at time $t$, see the proof of Lemma \ref{lem:mutn}, one can prove that the two distances have a cutoff phenomenon at the same time. Because of the regularity of our setting, we strongly believe that $d'_{N,\alpha}(t)= d^{\,\mathrm{sep}}_{N,\alpha}(t)$, however this is certainly not true for every reversible Markov chain.
\end{remark}

\subsection{Comments on the proof and related work}
The cutoff phenomenon is a widely studied topic in Markov chain theory, see \cite{diaconis1996cutoff,aldous2002reversible,LevPerWil} for an introduction. Thanks to recent remarkable efforts, many interesting examples are known of Markov chains exhibiting this particular type of phase transition. Unfortunately,  these seem to be mostly confined to the setting of Markov chains with discrete state space, see however \cite{HoughJiang17} for a recent exception. One of the reasons is possibly the fact that the analysis of total variation distance in a continuous state space is more demanding.  

\medskip
Let us briefly describe the main arguments used to establish the results of this article. 
Concerning the spectrum, the first observation is that when restricted to the invariant subspace of  linear functions, the generator can be easily diagonalised. In particular, one finds that  
$f_N$ is an eigenfunction of $\cL_{N,\alpha}$ with eigenvalue $\cos(\pi/N)-1$ and therefore, for all $\ga>0$,  the spectral gap 
is at most $1-\cos(\pi/N)$. Proving that this is actually the spectral gap is more delicate. We prove it only in the case $\ga\geq 1$ since we rely on the FKG inequality.
In this case, using the strict monotonicity of $f_N$, we show that any eigenfunction associated to an eigenvalue larger than $\cos(\pi/N)-1$ needs to be non-decreasing as well, and the FKG inequality is then applied to conclude that the only such eigenfunction is the constant one; see Section \ref{Sec:Eigen} below. Concerning the case $\alpha\in(0,1)$, we can only assert that the spectral gap is not smaller than $c N^{-2}$ for some constant $c=c(\alpha)>0$. Indeed, this is a direct consequence of  the mean field spectral gap \eqref{Eq:GapMF8} and a comparison argument (see e.g.\ \cite{Cap08}). 
%

\medskip
To establish the lower bound on the mixing time of Theorem \ref{Th:Maingen}, we use Wilson's method~\cite{Wil04}. For $\ga=1$, this was already used in \cite{RW05} to obtain the weaker lower bound
$$T_N(1/4)\geq c\,N^2\log N\;,$$
for some constant $c>0$. Here we sharpen this,  and obtain for any $\ga>0$ and for all $\gep>0$ $$T_{N,\alpha}(\gep)\geq \frac1{2\gap_N}(\log N - C_{\ga,\gep}),$$
see Proposition \ref{loow} below: note that we prove this result also for $\ga \in (0,1)$, but in this case $\gap_N$ is not defined as the spectral gap but simply as being $1-\cos(\pi/N)$. This suggests a cutoff window of order $N^2$, but we do not have a  corresponding upper bound. The proof of  the above lower bound is based on  Wilson's method with a careful choice of the initial condition $x$ and a comparison argument which allows us to control the variance of $f_N({\bf X}^x(t))$; see Section \ref{sec:lower}.
 
\medskip 
The proof of the upper bound on the mixing time of Theorem \ref{Th:Maingen} is the most involved part of the paper, and is worked out in Section \ref{sec:ub} and Section \ref{Sec:AbsCont}. Our strategy can be roughly outlined as follows; see Section \ref{sec:upperbound} below for a more detailed overview. 
In \cite{RW05}, and for $\alpha = 1$, it was shown that 
\begin{equation}\label{Eq:roughUB}
T_{N,\alpha}(1/4)\leq C N^2\log N,
\end{equation}
for some constant $C > 0$. This upper bound was obtained by estimating, under some monotone grand coupling, the hitting time of $0$ for the area comprised between the two extremal processes, that is, the processes starting from the highest ($x_1=\ldots=x_{N-1} = N$) and the lowest ($x_1=\ldots=x_{N-1} = 0$) initial conditions. Note that under such a coupling this hitting time bounds from above the coalescing time starting from any two initial conditions. The proof consisted of two steps that we now briefly recall. First, using the decay of the heat equation solved by the expectation of the area, one can show that after a time $C(\delta) N^2 \log N$, the area lies below $N^{-\delta}$ with large probability. Choosing $\delta$ large enough, the second step relies on a brute force argument that shrinks the area to $0$ in a time of order $\log N$ with large probability. In Appendix \ref{App:roughUB}, we give a proof of \eqref{Eq:roughUB} in the general case $\alpha > 0$: in contrast with the case $\alpha = 1$, there is no simple monotone  grand coupling for a general $\alpha$ that achieves an efficient coalescing time of the two extremal processes.
 We are then lead to controlling the coalescing time of the  stationary process and the process starting from some arbitrary initial condition.

The constant $C$ of \eqref{Eq:roughUB} is dictated by the smallest value $C(\delta)$ that one can take in the aforementioned strategy: since $\delta$ needs to be very large for the second step to apply, this value is much larger than desired. Our main contribution consists in introducing a sequence of intermediate steps that reduce the time necessary to bring the area to a small enough threshold from which a brute force argument can be applied. Namely, we use the first step above up to the target mixing time, which guarantees a shrinking of the expectation of the area down to $N^{3/2}$ (corresponding to $\delta = - 3/2$), and then present a coupling under which, through a sequence of successive stages, we are able to bring the area from $N^{3/2}$ down to $N^{-1}$ within a time of order $N^2$ with large probability, that is, a time negligible compared to the target mixing time. The last step is then a (slightly different) brute force argument that shrinks the area to $0$ with large probability within a time of order $\log N$. This program is carried out in Section \ref{sec:ub}. 

At a technical level, we use estimates on the derivative of the predictable bracket process associated to the evolution of the area together with diffusive estimates on supermartingales in order to upper bound the time needed for the area to descend through the intermediate values mentioned above. 
A key ingredient of these hitting time estimates, is the control of the gradients of the particle positions for the extremal processes, that is the process with maximal and minimal initial particle positions. The control of the gradients, in turn,  is obtained by coupling the extremal processes with the stationary distribution. To this end we need to establish, with an independent argument,  a sharp upper bound on the time needed for convergence to equilibrium of the 
two extremal processes.


The sharp upper bound for the extremal processes (by symmetry, we may  consider only the highest process) is performed in Section \ref{Sec:AbsCont}. This part of the proof is based on a strategy developed in the discrete setting by \cite{Lac16} for the adjacent transposition process. It relies on the FKG inequality and a version of the censoring inequality  from~\cite{PWcensoring}. Censoring inequalities are known to hold provided one has monotonicity of the density of the initial condition with respect to the invariant measure. While in the discrete setting this adapts well to the extremal process, in our context a nontrivial adaptation is required since the initial condition is singular with respect to the invariant measure.
%

%
%

\section{Preliminaries}
Let us consider a larger state-space where there is no constraint on $x_N$ 
\begin{equation}
 \gO^+_N:=\{ x=(x_1,\dots,x_{N})\in \bbR^{N} \ : \ 0\le x_1 \le x_2 \le \dots \le x_{N} \}.
\end{equation}
The generator $\cL_{N,\alpha}$ defines a dynamics also on this enlarged space.
If we let $$(\eta_i(x))^N_{i=1}:=(x_i-x_{i-1})_{i=1}^N$$ denote the increments of $x$, 
one can easily check that the distribution under which the increments $\eta_i$ are i.i.d.\
random variables with distribution $\Gamma(\alpha, \gl)$, $\gl>0$ being an arbitrary parameter, is a reversible measure for the dynamics. Here $\Gamma(\alpha, \gl)$ stands for the Gamma distribution with mean $\alpha / \gl$ and variance $\alpha / \gl^2$.

\begin{remark}\label{rem:betarev}
We remark that the Beta resampling rule \eqref{Eq:Beta} is a natural choice in our context. Indeed, suppose we define a generator as in \eqref{Eq:Generator2} with $\rho_\ga(u)du$ replaced by another probability $\nu(du)$ on the unit interval.
Then one can check that if $\mu$ is a probability measure on $\gO^+_N$ under which the increments $\eta_i$ are independent, the generator is reversible with respect to $\mu$ if and only if $\mu$ is the product of $\Gamma(\ga,\gl)$,  and $\nu(du)=\rho_\ga(u)du$, for some $\ga,\gl>0$. This rests on the well known fact that if $X,Y$ are two independent  random variables such that $X+Y$ and $X/(X+Y)$ are independent, then $X$ and $Y$ must have the gamma distribution, see \cite{lukacs1955}.
\end{remark}

%
%

\subsection{Monotone grand coupling}

There are two natural orders on $\gO^+_N$ associated with the dynamics.
The coordinate order 
\begin{equation}\label{coord}
x\ge x' \  \Leftrightarrow \ \left\{ \forall i \in \lint 1,N \rint, \  x_i\ge x'_i \right\},
\end{equation}
 and the gradient order 
 \begin{equation}\label{curlyorder}
 x\succcurlyeq x' \  \Leftrightarrow \ \left\{ \forall i \in \lint 1,N \rint, \  \eta_i(x)\ge \eta'_i(x') \right\}.
 \end{equation}
 Here and below we use the notation $\lint i,j\rint$ for the integer interval $[i,j]\cap\bbZ$. Note that the coordinate order is natural in both $\gO_N$ and $\gO_N^+$, while the gradient order is only relevant for the unconstrained space $\gO_N^+$. An important observation is that the Beta resampling dynamics preserves both orders in the following sense.

\begin{proposition}[Existence of a grand coupling] \label{grandgradient}
 For any $\alpha>0$, and $N\ge 2$, we can construct the trajectories $(\bX^x(t))_{t\ge 0}$ of the Markov chain on $\gO^+_N$ with generator $\cL_{N,\alpha}$ 
 on the same probability space in such a way that $\bbP$-a.s,
 for all $x,x'\in \gO^+_N$
 \begin{equation}
 \label{Eq:orders}
 \begin{split}
  x\ge  x' \  \Leftrightarrow \left\{ \forall t \ge 0,  \bX^x(t) \ge \bX^{x'}(t) \right\},\\
    x \succcurlyeq  x' \  \Leftrightarrow \left\{ \forall t \ge 0,  \bX^x(t) \succcurlyeq \bX^{x'}(t) \right\},\\
 \end{split}\end{equation}

\end{proposition}

\begin{proof}
 The coupling invoked above is the usual graphical construction (see e.g. \cite{Liggettbook}). To each $k\in\lint 1, N-1\rint$ we associate a Poisson clock process $(\cT^{(k)}_i)_{i\ge 1}$ whose increments are i.i.d.\ rate one exponentials,
 and a sequence $(U^{(k)}_i)_{i\ge 1}$ of i.i.d.\ symmetric Beta variables of paramenter $\alpha$. 
 Then, for all $x\in\gO^+_N$,  $(\bX^x(t))_{t\ge 0}$ is chosen to be càd-làg and 
 constant outside of the update times $(\cT^{(k)}_i)_{k\in \lint 1, N-1\rint, i\ge 1}$.
 At time $t=\cT^{(k)}_i$, if $U^{(k)}_i=u$  the $k$-th coordinate is updated as follows
 $$ X_k^x(t)= u X_{k-1}^x(t_{-})+(1-u) X_{k+1}^x(t_{-}) \text{ and } X_j^x(t):= X^x_j(t_{-}) \text{ for } j\ne k.$$ 
 One then checks by inspection that the process generated in this manner is the desired Markov chain. Moreover, the above construction implies that it preserves the two above mentioned orders in the sense of \eqref{Eq:orders}.
\end{proof}

Let us finally introduce the maximal $\wedge$ and minimal $\vee$ configurations for the coordinate order $\ge$ that we defined above:
\begin{equation}\label{Eq:WedgeVee}
\wedge_k = N\;,\quad \vee_k = 0\;,\quad \forall k\in \lint 1,N-1\rint\;.
\end{equation}

\subsection{The FKG inequality}\label{sec:fkg}

When $\alpha\ge 1$, the equilibrum measure 
$\pi_{N,\alpha}$ can be interpreted as a one dimensional gradient field associated with a convex potential.
More precisely  we can write

$$\pi_{N,\alpha}(\dd x):= C(N,\alpha) \exp\left(-\sum_{k=1}^{N} V(x_{i}-x_{i-1}) \right) \dd x,$$
where $V$ is the convex potential
\begin{equation}
V(u)= \begin{cases} +\infty \quad &\text{ if } u\le 0,\\
              - (\alpha-1) \log u \quad  &\text{ if } u>0.
             \end{cases}
\end{equation}
Below we use a standard application  of Holley's criterion (see \cite{Preston}) to show that $\pi_{N,\alpha}$ 
satisfies the so-called FKG inequality which entails positive correlation between increasing functions.

\medskip

We say that $f$ defined on $\gO_N$ is increasing if 
\begin{equation}\label{eq:increas}
 \forall x,x'\in \gO_N, \quad x\ge x' \Rightarrow f(x)\ge f(x').
 \end{equation}
A subset $A\subset \gO_N$ is said to be increasing if $\ind_A$ is an increasing function.
We let $\wedge$ and $\vee$ denote the following operations in $\gO_N$
\begin{equation}\label{minimax}
(x\vee x'):=(\max(x_i,x'_i))_{i=1}^{N-1} \;\text{  and  } \; (x\wedge x'):=(\min(x_i,x'_i))_{i=1}^{N-1}. 
\end{equation}
(Note that there is a little clash of notation with the maximal and minimal configurations introduced in \eqref{Eq:WedgeVee}: however we believe that this will never raise any confusion in the sequel).
\begin{proposition}[FKG inequality]\label{prop:fkg}
For any $\alpha\ge 1$ and $N\ge 2$, 
if $f$ and $g$ are increasing then 
$$\pi_{N,\alpha}(fg)\ge \pi_{N,\alpha}(f)\pi_{N,\alpha}(g).$$
Moreover, the inequality remains valid if $\pi_{N,\alpha}$ is replaced by its restriction $\pi_{N,\alpha}(\cdot|A)$ to any set $A$ which is stable under the operations $\vee$ and $\wedge$.
\end{proposition}

\begin{proof}
 Setting $H(x):=\sum_{k=1}^{N} V(x_{k}-x_{k-1})$, the density of 
 $\pi_{N,\alpha}$ with respect to Lebesgue is given by $e^{-H(x)}$ and 
 Holley's criterion shows that the inequality holds provided that 
 \begin{equation}\label{creat}
  H(x\vee x')+H(x\wedge x')\le  H(x)+H(x').
\end{equation}
The inequality \eqref{creat} can be deduced as a consequence of the convexity of $V$ see \cite[Appendix B1]{Giacomin}. For the measure restricted to a set $A$
it is sufficient to check that \eqref{creat} is valid when $H$ is replaced by
 \begin{equation}\label{defA}
  H_A(x):=\begin{cases} \sum_{k=1}^{N} V(x_{i}-x_{i-1}),  \quad &\text{if $x\in A$},\\
  +\infty, &\text{if $x\in A^{c}$},
                     \end{cases}
                     \end{equation}
                     which is an immediate consequence of \eqref{creat} under our assumption on $A$.
\end{proof}
Let us point out that the first inequality of Proposition \ref{prop:fkg} could be obtained directly from the monotonicity stated in Proposition \ref{grandgradient}, see for instance~\cite[Th 22.16]{LevPerWil}, and therefore does not require the convexity of the potential (that is, it holds also for $\alpha \in (0,1)$).

For two probability measures $\mu,\nu$ on $\Omega_N$, we write $\mu\leq \nu$ and say that $\mu$ is stochastically dominated by $\nu$ if for all increasing $f$ (in the sense of \eqref{eq:increas}) one has $\mu(f)\le\nu(f)$. 
Let us mention another application of Holley's criterion which we are going to use in the proof.

\begin{lemma}\label{lem:piapib}
 Let $A$ and $B$ be two increasing subsets of $\gO_N$ such that 
 $$\left\{ x\in A \ \text{ and } \  x'\in B\right\} \Rightarrow x \wedge x'\in B.$$
Then 
$$\pi_{N,\alpha}( \cdot \ | \ A)\geq \pi_{N,\alpha}( \cdot \ | \ B).$$
\end{lemma}
\begin{proof}
 Let $\mu_A:=\pi_{N,\alpha}( \cdot \ | \ A)$
 and $\mu_B:=\pi_{N,\alpha}( \cdot \ | \ B)$.
 These probability measures have density proportional to 
 $\exp(-H_A(x))$ and $\exp(-H_B(x))$ respectively, where the potentials $H_A,H_B$ are defined as in \eqref{defA}. The result then directly follows from \cite[Proposition 1]{Preston} if one can show that for every $x,x'\in \gO_N$
\begin{equation}
  H_A(x\vee x')+H_B(x\wedge x')\le  H_A(x)+H_B(x').
\end{equation}
This in turn follows from the inequality \eqref{creat} for $H$ and our assumption on $A,B$.
 
\end{proof}

\subsection{Identification of the spectral gap 
}\label{Sec:Eigen}
Here we prove the first statement of Theorem \ref{Th:Maingen}.
Fix $\ga\geq 1$ and $N\geq 2$, and write $\cL$ for the generator $\cL_{N,\ga}$. 
Using the expression \eqref{Eq:Generator2}, the action of the generator on the coordinate map $h_k(x):=x_k$, $k=1,\dots, N-1$, is given by
\begin{equation}\label{genlinear}
(\cL h_k)(x) = \frac12 \Delta h_k(x)\;,
\end{equation}
where $\Delta$ denotes the discrete Laplace operator
$$ \Delta x_k := x_{k+1} - 2 x_k + x_{k-1}\;.$$
Summation by parts and \eqref{genlinear} 
then shows that for every $j\in \lint 1,N-1\rint$ the map
\begin{equation}\label{defj}
f_N^{(j)} (x) := \sum_{k=1}^{N-1} \sin\left( \frac{j \pi k}{N}\right) (x_k-k)\;,
\end{equation}
is an eigenfunction of $\cL$  with the eigenvalue $-\lambda_N^{(j)}$ where
$$ \lambda_N^{(j)} := 1 - \cos\left( \frac{j\pi}{N}\right)\;.$$
In the case $j=1$, we simply write $f_N$ for $f_N^{(1)}$ and $\gl_N$ for $ \lambda_N^{(1)}$. We now turn to   show that $\gl_N$ is actually the spectral gap.

It is not hard to check that for any $n\ge 1$, the set of all polynomials of degree at most $n$ in the variables $x_1,\ldots,x_{N-1}$ is stable under the action of the generator.
When restricted to any such set, the generator admits a finite complete  decomposition into eigenvalues / eigenfunctions. By density of polynomials in $L^2(\gO_N,\pi_{N,\alpha})$, there exists an orthonormal basis of polynomial eigenfunctions and therefore the generator has pure point spectrum in $L^2(\gO_N,\pi_{N,\alpha})$.

$\cL$ maps to zero all constant functions and any nontrivial eigenvalue of $\cL$ must be associated with an eigenfunction with mean zero. Therefore it is sufficient to show that if $g$ is a normalized polynomial eigenfunction such that $\pi_{N,\alpha}(g)=0$ and $\cL g=-\mu g$, with $\mu < \lambda_N$, then $g=0$. Since $g$ is polynomial, and since $f_N$ is strictly increasing in all its variables, there exists $\epsilon > 0$ such that $f_N + \epsilon g$ remains increasing in all its variables. Next, we define the following normalized function
$$ v_t := \frac{P_t (f_N + \epsilon g)}{\|P_t (f_N + \epsilon g)\|_2 }\;,$$
where $P_t=e^{t\cL}$ denotes the semigroup generated by $\cL$, and $\|\cdot\|_2$ is the $L^2(\gO_N,\pi_{N,\alpha})$-norm.
From our assumptions one has $\pi_{N,\alpha}(f_Ng)=0$ and 
$$
P_t (f_N + \epsilon g) = e^{-\gl_Nt}f_N + \epsilon \,e^{-\mu t}g\,,
$$ with $\mu<\gl_N$. It follows that $v_t \to g$ as $t\to\infty$ in $L^2(\gO_N,\pi_{N,\alpha})$. On the other hand, the semigroup preserves monotonicity by Proposition \ref{grandgradient}, so that $v_t$ is a non-decreasing function at any time $t$.
Thus, $g$ must be also non-decreasing. Notice that so far the argument is valid for any $\ga>0$. We shall now use the assumption $\ga\geq	1$. 
 
 The FKG inequality  and the orthogonality of $f_N$ and $g$  imply that the centered coordinate maps $\bar h_k(x):=x_k -k$ satisfy
$$
\pi_{N,\alpha}(\bar h_k g)=0,\qquad k\in\lint 1,N-1\rint.
$$
Indeed, by Proposition \ref{prop:fkg} one has $\pi_{N,\alpha}(\bar h_k g)\geq 0$ and if this is positive for some $k$, then $\pi_{N,\alpha}(f_N g)>0$.

Let $A_{\gep}$ be the event $ h_1\ge N-\gep$, with $\gep\in(0,1)$. Since both $A_\gep$ and $A_\gep^\cc$ are stable for the operations $\wedge$ and $\vee$ introduced in \eqref{minimax}, by Proposition \ref{prop:fkg} the restrictions $\pi_{N,\alpha}(\cdot \ | \ A_{\gep})$ and $\pi_{N,\alpha}\left(\cdot \ | \ A^\cc_{\gep}\right)$ satisfy the FKG inequality.
Therefore,
\begin{align}
 0=\pi_{N,\alpha}\left(\bar h_1g\right) &=\pi_{N,\alpha}\left(A_{\gep}\right)\pi_{N,\alpha}\left(\bar h_1 g \ | \ A_{\gep}\right)+\pi_{N,\alpha}\left(A_{\gep}^{\cc}\right)\pi_{N,\alpha}\left(\bar h_1 g \ | \ A_{\gep}^{\cc}\right)\nonumber\\
 &\ge \pi_{N,\alpha}\left(\bar h_1 \ind_{A_{\gep}}\right)\pi_{N,\alpha}\left(g \ | \ A_{\gep}\right)+ \pi_{N,\alpha}\left(\bar h_1 \ind_{A_{\gep}^{\cc}}\right)\pi_{N,\alpha}\left(g \ | \ A_{\gep}^{\cc}\right).
\label{Eq:fkgh}
\end{align}
The FKG inequality implies that both terms in the last expression are nonnegative: indeed, $A_\gep$ is increasing and therefore \begin{align*}
\pi_{N,\alpha}\left(\bar h_1 \ind_{A_{\gep}}\right)&\geq \pi_{N,\alpha}\left(\bar h_1\right) \pi_{N,\alpha}\left(A_{\gep}\right)=0\;,\\
\pi_{N,\alpha}\left(g \,\ind_{A_{\gep}}\right)&\geq \pi_{N,\alpha}\left(g\right) \pi_{N,\alpha}\left(A_{\gep}\right)= 0,\\
\pi_{N,\alpha}\left(\bar h_1 \ind_{A_{\gep}^{\cc}}\right)&\leq \pi_{N,\alpha}\left(\bar h_1\right) \pi_{N,\alpha}\left(A_{\gep}^{\cc}\right)=0\;,\\ \pi_{N,\alpha}\left(g \,\ind_{A_{\gep}^{\cc}}\right)&\leq \pi_{N,\alpha}\left(g\right) \pi_{N,\alpha}\left(A_{\gep}^{\cc}\right)= 0.
 \end{align*}
Hence the first term in the right hand side of \eqref{Eq:fkgh} must be zero. Since $\bar h_1 \ind_{A_{\gep}}\geq N-1-\gep>0$, this is possible only if $\pi_{N,\alpha}(g \, |  A_{\gep})=0$.
However, that and the continuity of $g$ imply that the extremal configuration $x_i\equiv N$ satisfies
$$g( N,\dots,N)= \lim_{\gep\to 0^+}\inf_{x\in A_{\gep}} g(x)=0 .$$
Hence $g(x)\leq 0$ for all $x\in\gO_N$ and therefore $g\equiv 0$.
\qed

\subsection{Absolute continuity}
Recall that under $\pi_{N,\alpha}$, the r.v.~$(\eta_k)_{k=1}^N$ sum up to $N$. At several places in the proofs, it will be convenient to deal with independent r.v.~instead. To that end, we use the following informal fact: for any $p\in (0,1)$, the law of $(\eta_1,\eta_2,\ldots,\eta_{\lfloor pN \rfloor})$ under $\pi_{N,\alpha}$ is uniformly over $N$ absolutely continuous w.r.t.~the law of $\lfloor pN \rfloor$ independent $\gG(\alpha,\alpha)$ r.v. The formal version is stated below:
\begin{lemma}\label{Lemma:AbsCont}
Fix $p\in (0,1)$. There exists a constant $C_p>0$ such that for all $N\ge 1$, writing $n=\lfloor pN\rfloor$, for any bounded measurable function $f:\bbR^n \to \bbR_+$ we have
$$ \pi_{N,\alpha}[f(\eta_1,\ldots,\eta_n)] \le C_p\,\nu_N[f(\eta_1,\ldots,\eta_n)]\;,$$
where $\nu_N$ is the law on $\Omega_N^+$ under which the $\eta_k$'s are IID~$\gG(\alpha,\alpha)$ r.v.
\end{lemma}
\begin{proof}
Let $g(\cdot)$ be the density function of a centered Gaussian distribution of unit variance. Let $q_k$ be the density function of a \emph{centered} $\gG(k\alpha,\alpha)$ r.v. By the Local Limit Theorem~\cite[Th.VII.2.7]{Petrov}, we have
\begin{align*}
\lim_{N\to \infty}\sqrt{\frac{N}{\alpha}} q_N(0) - g(0) = 0,\\
\lim_{N\to \infty}\sup_{y\in\bbR} \Big|\sqrt{\frac{N-n_N}{\alpha}} q_{N-n_N}\Big(y\sqrt{\frac{N-n_N}{\alpha}}\Big) - g(y)\Big|=0\;.
\end{align*}
 Using the above together with the fact that $g$ is maximized at $0$, we have for all $N$ sufficiently large 
\begin{multline*}
\pi_{N,\alpha}[f(\eta_1,\ldots,\eta_n)]= \frac{\nu_N[f(\eta_1,\ldots,\eta_n) q_{N-n}(n-x_n)]}{q_N(0)}\le \frac{2 \nu_N[f(\eta_1,\ldots,\eta_n)]}{\sqrt{1-p}}  \;.
\end{multline*}
The result follows by tuning the value of $C_p$ to also include the first few values of $N$.
\end{proof}

\section{Lower bound on total variation distance}\label{sec:lower}
In this section we prove the lower bound on the mixing time
displayed in Theorems \ref{Th:Main} and \ref{Th:Maingen}.
We obtain in fact a more quantitative lower bound which is valid for all values $\alpha>0$. 

\begin{proposition}\label{loow}
 For any $\alpha>0$, for any $\gep\in(0,1)$ there exists $C_{\alpha,\gep}>0$
 such that for all $N\ge 2$
 \begin{equation}\label{Eq:lowbo}
  T_{N,\alpha}(1-\gep)\ge \frac{N^2}{\pi^2}\left(\log N-C_{\alpha,\gep}\right).
 \end{equation}
\end{proposition}
We are going to follow a variant of the method introduced by Wilson in \cite{Wil04}:
to obtain the lower bound \eqref{Eq:lowbo}, we select a specific test function $f$ and show that by time $t=\frac{N^2}{\pi^2}(\log N-C_{\alpha,\gep})$, the value $f({\bf X}(t))$ is far from the equilibrium value $\pi_{N,\alpha}(f)$ with large probability. This is achieved by picking a suitable initial condition and by evaluating the first two moments of $f({\bf X}(t))$. As in the case of the exclusion process, we select  $f=f_N$, the eigenfunction appearing in Proposition \ref{gapz}. As noted in the proof of that proposition, $f_N$ is an eigenfunction with eigenvalue $\gap_N$ for all $\ga>0$. 

\medskip

The initial condition $\bX(0)$ is defined as follows. When $\alpha=1$ we let $\eta_k, k=1,\ldots,N/2$ be\footnote{In this section, we write $N/2$ for $\lfloor N/2 \rfloor$ in order to alleviate notations.} i.i.d.\ exponential random variables with mean $2$ 
 conditioned on $\eta_1+\ldots+\eta_{N/2} = N$, and set $\eta_k = 0$ for all $k\in \lint N/2 +1,N\rint$. More generally, for arbitrary $\alpha >0$ we choose the distribution of $\eta_k$ for $k\le N/2$ to be i.i.d.\ $\gG(\alpha,\alpha/2)$, with the same conditioning on the sum, and set again $\eta_k = 0$ for all $k\in \lint N/2 +1,N\rint$. For the rest of this section, we let ${\bf X}(t)$ be the process starting from the random initial condition with increments $\eta_k$,
 \begin{equation}\label{eq:init}
 X_k(0) = \sum_{i=1}^k \eta_i\;,\quad k\in \lint 1,N\rint\;.
 \end{equation}
We shall use $\bbP,\bbE$ for the corresponding probability measure and expectation, and $\var$ for the associated variance.
For later use we also prove a result concerning the variance of other Fourier coefficients of $\bX(t)$, namely the functions $f_N^{(j)}$ from \eqref{defj}.
\begin{lemma}\label{Lemma:Wilson}
For the process described above with initial condition \eqref{eq:init},
there exists $C>0$ such that for all $t\ge 0$ and all $N$ large enough
$$ \bbE[f_N({\bf X}(t))] \ge C^{-1} N^2 e^{-\gap_N t}\,,\;\text{ and }\;  
\var[f_N(\bX(t))] \le C N^3\,,
$$
and for every $j\in \lint 2,N-1\rint$
$$ 
\var[f^{(j)}_N(\bX(t))] \le C j^{-2}N^3\,,
$$
\end{lemma}
\noindent With this lemma at hand, the proof of the lower bound is simple.
\begin{proof}[Proof of Proposition \ref{loow}]
Define $$B_t=\left\{x\in\gO_N:\, f_N(x)\geq \tfrac12 \,\bbE[f_N({\bf X}(t))]\right\},$$ 
and note that for all $t\geq 0$:
$$d^N(t)\ge \| \bbP( \bX(t)\in \cdot) -\pi_{N,\alpha} \|_{TV}\ge \bbP( \bX(t)\in B_t) - \pi_{N,\alpha}(B_t).$$
Chebyshev's inequality and Lemma \ref{Lemma:Wilson} imply that 
$$
\bbP( \bX(t)\in B_t) \geq 1 - 4C^3N^{-1} e^{2\gap_N t}.
$$
Similarly, noting that $\var_{\pi_{N,\alpha}}(f_N)=\pi_{N,\alpha}(f_N^2)=\lim_{t\to\infty}\var[f_N(\bX(t))]$, one has
$$
\pi_{N,\alpha}(B_t) \leq  4C^3N^{-1} e^{2\gap_N t}.
$$
Recalling that $\gap_N=1-\cos(\pi/N)$ concludes the proof.
\end{proof}

We close this section with the proof of the lemma.

\begin{proof}[Proof of Lemma \ref{Lemma:Wilson}] 
First of all, since $f_N$ is an eigenfunction of $\cL_{N,\alpha}$ associated with the eigenvalue $\gap_N$, we have
$$ \bbE[f_N(\bX(t))] =  \bbE[f_N(\bX(0))]\,e^{-\gap_N t} \ge \tfrac{N^2}{20} \,e^{-\gap_N t}\;,$$
where the last bound follows from the fact that the initial condition satisfies 
$$\bbE[X_k(0)]\geq \min\{2k,N\}.$$
To compute the variance at a fixed time $t > 0$ (we focus on the case $j=1$ to keep the notation light and explain the general case at the end of the proof), we introduce the process 
$$ M_s^t := e^{\gap_N(s-t)} f_N(\bX(s)) 
\;,\quad s\in[0,t]
\;.$$
Since $\bbE[f_N(\bX(s))|\bX(u)] = e^{-\gap_N (s-u)}f_N(\bX(u))$, $u\in [0,s]$, it follows that $(M_s^t,\,s\in[0,t])$ is a martingale.  The associated increasing predictable process, or angle bracket, is denoted $\langle M^t_\cdot \rangle_s$, $s\in[0,t]$. In particular, we look for an upper bound on $\var[f_N(\bX(t))]=\E[\langle M^t_\cdot \rangle_t]$. 

Note that independently of the value of $\alpha$, when an update is performed at coordinate $k$, the value of   
$f_N({\bX}(t))$ varies at most by $\eta_k(s) + \eta_{k+1}(s)$, where $\eta_k(s)=(X_k-X_{k-1})(s)$.
Thus, 
\begin{align*}
\partial_s \langle M^t_\cdot \rangle_s
&\le e^{2\gap_N(s-t)} \sum_{k=1}^{N-1} (\eta_k(s) + \eta_{k+1}(s))^2\\&
\le 4 e^{2\gap_N(s-t)} \sum_{k=1}^{N} (\eta_k(s))^2.
\end{align*}
Therefore,
\begin{equation}\label{Eq:vafxt}
 \var[f_N(\bX(t))]=\E[\langle M^t_\cdot \rangle_t]
 \le  4\int_0^t e^{2\gap_N(s-t)} \sum_{k=1}^{N}\E\left[\eta_k(s)^2\right] \dd s
\end{equation}

%

Now using monotonicity of the dynamics for the gradients (Proposition \ref{grandgradient}), we can replace $\eta_k(s)$ by the gradients corresponding to a dynamics on $\gO^+_N$ starting from a larger initial condition for the order $\succcurlyeq$ defined in \eqref{curlyorder}.
A natural choice is to pick an initial condition which is stationary for the dynamics so that the dependence in $s$ vanishes in the integral.

\medskip

Let $(\eta'_k)_{k=1}^N$ be i.i.d.\ variables with distribution $\gG(\alpha, \alpha/2)$, 
conditioned to  $$\sum_{j=1}^{N/2} \eta'_j\ge N,$$ and consider $\bX'(t)$ the dynamics 
on $\gO^+_N$ with initial condition $X'_k(0):=\sum_{i=1}^k \eta'_k$.
Note that one can construct $\eta$ and $\eta'$ on the same probability space in such a way that $\eta_k\le \eta'_k$ for all $k$ by setting 
\begin{equation}\label{Eq:eta_coupling}
\eta_k:=\frac{N\,\eta'_k}{\sum_{j\le N/2} \eta_j'}\,, \;\text{ for } k\le N/2,
\end{equation}
and $\eta_k=0$, for $k\in\lint N/2+1,N\rint$. Indeed, by a standard property of the gamma distribution,
the variables $\eta'_k/\sum_{j\le N/2} \eta_j'$ and $\sum_{j\le N/2} \eta_j'$ are independent and therefore the $\eta_k$ defined in \eqref{Eq:eta_coupling} has the correct distribution. 
Hence using Proposition \ref{grandgradient} we can assume $\bX(t)\preccurlyeq \bX'(t)$ for all $t\ge 0$ and thus  in particular 
$\bbE\left[\eta_k(s)^2\right] \le 
\bbE\left[\eta'_k(s)^2\right]$.

\medskip

Finally let $(\eta''_k)_{k=1}^N$ be (unconditioned) i.i.d.\  with distribution $\gG(\ga,\ga/2)$. Defining the dynamics $\eta''_k(s)$ with this initial condition one has
\begin{align*}
 \bbE\left[\eta'_k(s)^2\right]&= \bbE\left[\eta''_k(s)^2 \ | \ \textstyle{\sum_{j\le N/2}} \,\eta_j''(0) \ge N \right]
 \le \frac{\bbE\left[\eta''_k(s)^2\right]}{\bbP\left[ \sum_{j\le N/2} \eta_j''(0)\ge N\right]}.
\end{align*}
By stationarity,
$$
\bbE\left[\eta''_k(s)^2\right]=\bbE\left[\eta''_k(0)^2\right]=4+\frac{4}{\ga}.
$$
Since the expected value of each $\eta''_k(0)$ is $2$, the central limit theorem shows that
the event $\sum_{j\le N/2} \eta_j''(0)\ge N$ has probability at least $1/3$ if $N$ is sufficiently large.
In conclusion,
$$
\bbE\left[\eta_k(s)^2\right] \le \bbE\left[\eta'_k(s)^2\right]\leq 12(1+\ga^{-1})\,.
$$
Using this bound in \eqref{Eq:vafxt} shows that $\var[f_N(\bX(t))]\leq CN/\gap_N\le C' N^3$ uniformly in $t\geq 0$. 

\medskip

\noindent
For $j\in \lint 2,N-1\rint$, we repeat the above procedure for the martingale 

$$M^{t,j}_s:= e^{-\gl^{(j)}_{N}(s-t)} f_N^{(j)}(\bX(t))$$
and obtain  $\var[f^{(j)}_N(\bX(t))]\leq CN/\gl^{(j)}_N \leq C N^3 j^{-2} $.
This concludes the proof of Lemma \ref{Lemma:Wilson}.

\end{proof}
\section{Mixing time for the separation distance}\label{sec:sepa}
Here we prove Theorem \ref{Th:Separation}. The main result of this section is the following lower bound.

\begin{proposition}\label{Prop:lowbosep}
  For any $\alpha>0$, there exists $C_{\alpha,\gep}>0$ such that for any $\gep\in(0,1)$ and for all $N\ge N_{\gep}$ sufficiently large we have
 \begin{equation}\label{Eq:lowbosep}
  T^{\,\mathrm{sep}}_{N,\alpha}(\gep)\ge N^2\left(\frac{2}{\pi^2}\log N-C_{\alpha,\gep}\log \log N \right).
 \end{equation}
\end{proposition}

With this result at hand, and assuming the validity of Theorem \ref{Th:Maingen}, the derivation of the asymptotic of the separation mixing times is somewhat standard.
\begin{proof}[Proof of Theorem \ref{Th:Separation}]
Proposition \ref{Prop:lowbosep} gives the desired lower bound on the mixing times. Regarding the upper bound, we adapt the argument used in the discrete setup (see e.g \cite[Lemma 19.3]{LevPerWil}) to show that 
\begin{equation}\label{lapreuve}
d^{\,\mathrm{sep}}_{N,\alpha}(2t)\le 1-(1-2d_{N,\alpha}(t))_+^2,
\end{equation}
where $(\cdot)_+$ denotes the positive part. Theorem \ref{Th:Maingen} and \eqref{lapreuve} clearly imply the desired upper bound in Theorem \ref{Th:Separation}. 

Recalling the notation $P_t(A,B)$ defined below \eqref{defseparation} we notice that reversibility implies that $P_t(A,B)=P_t(B,A)$ and as a consequence 
\begin{equation}\label{reversibb}
P_t(B,\dd z)=P^z_t(B)\pi_{N,\alpha}(\dd z)
 \end{equation}
Thus for any $A,B$ with positive Lebesgue measure, using the semigroup property at the first line and reversibility at the second line
\begin{align*}
P_{2t}(A,B) &= \int_{\gO_N}  P_{t}(A,\dd z) P^z_t(B)\\
&= \int_{\gO_N}  P^z_{t}(A) P^z_t(B) \pi_{N,\alpha}(\dd z) 
\end{align*}
By Schwarz' inequality, we have
\begin{align*}
\sqrt{\frac{P_{2t}(A,B)}{\pi_{N,\alpha}(A)\pi_{N,\alpha}(B)}} &\ge  \int_{\gO_N}  \sqrt{\frac{P^z_{t}(A) P^z_t(B)}{\pi_{N,\alpha}(A)\pi_{N,\alpha}(B)}} \pi_{N,\alpha}(\dd z) \\
&\ge \int_{\gO_N}  \frac{P^z_{t}(A)}{\pi_{N,\alpha}(A)} \wedge \frac{P^z_{t}(B)}{\pi_{N,\alpha}(B)}\pi_{N,\alpha}(\dd z).
\end{align*}
Let $\bar P_t(A,\dd z)= (\pi_{N,\alpha}(A))^{-1}P_t(A,\dd z)$ denote the normalized version of $P_t(A,\dd z)$. 
From \eqref{reversibb} it follows that $\frac{P^z_{t}(A)}{\pi_{N,\alpha}(A)}$ and $\frac{P^z_{t}(B)}{\pi_{N,\alpha}(B)}$ are the respective densities of $\bar P_t(A,\cdot)$ and $\bar P_t(B,\cdot)$ w.r.t. $\pi_{N,\alpha}$. Therefore, using the triangular inequality one has
\begin{align*}
\sqrt{\frac{P_{2t}(A,B)}{\pi_{N,\alpha}(A)\pi_{N,\alpha}(B)}} &\ge 
1- \| \bar P_t(A,\cdot)- \bar P_t(B,\cdot)\|_{TV} 
\ge (1-2d_{N,\alpha}(t))_+.
\end{align*} 
Taking the infimum over $A$ and $B$ yields \eqref{lapreuve}.
%
\end{proof}

\begin{proof}[Proof of Proposition \ref{Prop:lowbosep}] 
We are going  to show that for all $\alpha,\gep>0$, there exists $C>0$ such that if  
$$t_{0}:= N^2\left(\frac{2}{\pi^2}\log N-C\log \log N \right), $$
then one can find measurable sets $A,B\subset \gO_N$ such that  
$$ P_{t_{0}}(A,B)\le \gep \,\pi_{N,\alpha}(A)\pi_{N,\alpha}(B).$$ 
A natural choice to minimize $P_{t_{0}}(A,B)$ is to choose $A,B$ as tiny neighbourhoods of the opposite extremal configurations. 
%
We define the neighbourhoods of our extremal configurations $\vee$ and $\wedge$ as follows
\begin{equation}
\begin{split}
 \bar\wedge&:= \{ x\in \gO_N \ : \ \forall i \in \lint 1,N-1 \rint, x_i\ge N-1/N\},\\
 \underline\vee&:= \{ x\in \gO_N \ : \ \forall i \in \lint 1,N-1 \rint, x_i\le 1/N\}.
\end{split}
\end{equation}
and we are going to prove 
\begin{equation}\label{toprove}
\int_{\bar\wedge}P_{t_0}(x,\underline\vee)\pi_{N,\alpha}(\dd x)\le \gep \pi_{N,\alpha}(\bar\wedge)\pi_{N,\alpha}(\underline\vee).
\end{equation}
We assume for simplicity that $N$ is even just for the sake of notation. The first step of the proof is to reduce \eqref{toprove} to a simple statement about the process at time $t_0/2$

\begin{equation}\label{simplest}
 \frac{1}{\pi_{N,\alpha}(\bar\wedge)}\int_{\bar\wedge} P^x_{t_0/2}[ x_{N/2} < N/2] \pi_{N,\alpha}(\dd x) \le \gep/4,
\end{equation}
and the second step is to prove \eqref{simplest}.

\medskip

Let us now show that 
\eqref{toprove} follows from \eqref{simplest}
as a  consequence of reversibility and the FKG inequality.
The proof is in fact very similar to the one developped in \cite[Section 7.1]{Lac16} in a discrete setup.
We have 
\begin{multline}\label{kernex}
 \int_{\bar\wedge}P_{t_0}(x,\underline\vee)\pi_{N,\alpha}(\dd x) \\
 =\int_{x\in \bar\wedge} \int_{y\in \gO_N}  \int_{z\in \underline\vee}  \pi_{N,\alpha}(\dd x) P_{t_0/2}(x,\dd y)P_{t_0/2}(y,\dd z)   
\end{multline}
Now we can split the integral on $y$ into two contributions 
$$ U_+=\{ x\in \gO_N \ : x_{N/2}\ge N/2\} \quad \text { and } \quad   U_-=\{ x\in \gO_N \ : x_{N/2}\le N/2\},$$
with $U_+\cup U_-=\gO_N$ and $\pi_{N,\alpha}(U_+\cap U_-)=0$.
Using reversibility and symmetry we have 
\begin{multline}
 \int_{x\in \bar\wedge} \int_{y\in U_+}  \int_{z\in \underline\vee}  \pi_{N,\alpha}(\dd x) P_{t_0/2}(x,\dd y)P_{t_0/2}(y,\dd z) \\
 =  \int_{x\in \bar\wedge} \int_{y\in U_+}  \int_{z\in \underline\vee}  \pi_{N,\alpha}(\dd z) P_{t_0/2}(y,\dd x)P_{t_0/2}(z,\dd y) \\
  =  \int_{x\in \underline\vee} \int_{y\in U_-}  \int_{z\in \bar\wedge}  \pi_{N,\alpha}(\dd z) P_{t_0/2}(y,\dd x)P_{t_0/2}(z,\dd y), 
\end{multline}
so that the right hand side of \eqref{kernex} is equal to 
\begin{equation}\label{zoop}
 2\int_{x\in \bar\wedge} \int_{y\in U_-}  \pi_{N,\alpha}(\dd x) P_{t_0/2}(x,\dd y)P_{t_0/2}(y,\underline\vee).
\end{equation}
Now according to the observations made in Section \ref{Sec:AbsCont} (see Lemma \ref{lem:muqi}), the measure
$$\int_{x\in \bar\wedge} \pi_{N,\alpha}(\dd x) P_{t_0/2}(x,\dd y)$$ has an increasing density with respect to $\pi_{N,\alpha}$ (call it $\rho_+$).
Also we observe that by monotone coupling $P_{t_0/2}(y,\underline\vee)$ is decreasing in $y$.
Hence using the FKG inequality (Proposition \ref{prop:fkg} for the restriction to the stable set $U_-$) we obtain that the 
quantity \eqref{zoop} satisfies 
\begin{multline}
 2 \int_{U_-} \rho_+(y)P_{t_0/2}(y,\underline\vee)\pi_{N,\alpha}(\dd y)
 \\
 \le \frac{2}{\pi_{N,\alpha}(U_-)} \int_{U_-} \rho_+(y) \pi_{N,\alpha} (\dd y)
  \int_{U_-} P_{t_0/2}(y,\underline\vee) \pi_{N,\alpha} (\dd y).
\end{multline}
Using stationarity we see that the second integral is smaller than
$$ \int_{\gO_N} P_{t_0/2}(y,\underline\vee) \pi_{N,\alpha} (\dd y)=\pi_{N,\alpha}(\underline\vee).$$
Hence using the fact that $\pi_{N,\alpha}(U_-)=1/2$ and replacing $\rho_+(x)\pi_{N,\alpha}(\dd x)$ by its definition we can conclude that 
\begin{equation}
 \int_{\bar\wedge}P_{t_0}(x,\underline\vee)\pi_{N,\alpha}(\dd x)\le 4  \pi_{N,\alpha}(\underline\vee) \int_{\bar\wedge}P_{t_0/2}(x,U_-)\pi_{N,\alpha} (\dd x). 
\end{equation}
This proves that \eqref{toprove} follows from \eqref{simplest}. 
It remains to prove \eqref{simplest},
or in other words that starting from a random initial condition in $\bar\wedge$
the probability that $\bX_{t_0/2}\in U_-$ is small.


\medskip

\noindent Note that \eqref{simplest} follows from a slightly stronger result for the maximal initial condition
\begin{equation}\label{trook}
\bbP[X^{\wedge}_{N/2}(t_0/2)\ge N/2+1] \ge 1- \gep/4.
\end{equation}
Indeed, with the grand coupling described in the proof of Proposition 
\ref{grandgradient}, for any $x\in \bar\wedge$, $i\in \lint 1, N-1\rint$ and any $t>0$
we have $X^x_i(t)\ge X^{\wedge}_{i}-1.$

\medskip

Now by monotonicity we can replace $\bX^{\wedge}(t)$ by the dynamics with the random initial condition considered in Lemma \ref{Lemma:Wilson}.

\medskip

\noindent The function $i \mapsto X_{i}(t)-i$ can be decomposed on the orthonormal basis of the discrete Laplacian (endowed with Dirichlet b.c.). Using \eqref{defj}, we thus obtain
\begin{equation}\label{Eq:XiFourier}
 X_{i}(t)-i=\frac{2}{N}\sum_{j=1}^{N-1} \sin\left( \frac{ij\pi}{N}\right)
f^{(j)}_N(\bX(t)).
\end{equation}
Using the fact that $\bbE[f^{(j)}_N(\bX(t))]=e^{-\gl^{(j)}_{N} t}\bbE[f^{(j)}_N(\bX(0))]$,  and observing that all the terms with $j\ge 2$ become negligible, it follows that 
if  the constant $C$ is large enough, then
\begin{equation}\label{fm}
 \bbE[X_{N/2}(t_0/2)-N/2]\ge N^{1/2} (\log N)^2
\end{equation}
Now the important part is to control the variance of $X_{N/2}(t)$. Our control is uniform in time.

From the very rough bound $\var (\sum_{i\in I} Z_i)\le\left( \sum_{i\in I}\sqrt{\var Z_i}\right)^2$ valid for any sequence of random variables $(Z_i)_{i\in I}$, and  using Lemma \ref{Lemma:Wilson} one obtains
\begin{equation}\label{sm}
\var( X_{N/2}(t))\le 4 N^{-2} \left( \sum_{j=1}^{N-1}\sqrt{\var f^{(j)}_N(\bX(t))}\right)^2 \le C N (\log N)^2.
\end{equation}
Then \eqref{trook} can easily be deduced by combining the inequalities for  the two first moments \eqref{fm} and \eqref{sm}.
%
\end{proof}
\section{Upper bound on total variation distance}\label{sec:ub}

\subsection{Decomposition of the proof}\label{sec:upperbound}
In this section we prove the upper bound on the total variation mixing time 
displayed in Theorems \ref{Th:Main} and \ref{Th:Maingen}.
Given $\delta>0$ we set
$$ t_{N,\delta}=t_{\delta}:= (1+\delta) \frac{\log N}{2 \gap_N}\;,$$
and we are going to prove that
\begin{equation}\label{eq:thth}
 \lim_{N\to \infty} \sup_{x\in \gO_N} \| P_{t_{N,\delta}}^x - \pi_{N,\alpha} \|_{TV}=0\;.
\end{equation}
Since $\delta > 0$ is arbitrary and $\gap_N
{\sim} \frac{1}{2}\pi^2 N^{-2}$ this yields the desired upper bound.
We establish this statement via two intermediate propositions. Firstly, we show that by time $t_{\delta}$ we can bring to equilibrium the extremal process started in the maximal configuration $\wedge$.
\begin{proposition}\label{Prop:AbsCont}
For any $\delta>0$ we have
$$\lim_{N\to \infty} 
\| P_{t_{N,\delta}}^\wedge - \pi_{N,\alpha} \|_{TV} = 0\;.$$
\end{proposition}
Secondly, we show that with large probability and for any given initial condition $x$, we can couple the process starting form $x$ with the maximal process.
\begin{proposition}\label{Prop:Coalesce}
For any $\delta>0$ we have
$$\lim_{N\to \infty} \sup_{x\in \Omega_N} \|P^x_{t_{N,\delta}}-
P^{\wedge}_{t_{N,\delta}}\|_{TV}=0.$$
\end{proposition}

By the triangle inequality, \eqref{eq:thth} is an immediate consequence of the two propositions above. 

\medskip

Even though Proposition \ref{Prop:AbsCont} is a consequence of Proposition \ref{Prop:Coalesce}, our proof of the latter will actually rely on the former result; see also Remark \ref{rem:mono} below. Consequently, we will need an independent proof of Proposition \ref{Prop:AbsCont} and this will be carried out in Section \ref{Sec:AbsCont}. The remainder of this section is concerned with the proof of Proposition \ref{Prop:Coalesce}.

\medskip

From now on, $x\in \Omega_N$ is fixed; however all the constants that will appear below will be \emph{independent} of this chosen $x$. In the forthcoming Section \ref{thecoupling}, we construct the processes
$\bX^\wedge$ and  $\bX^x$ on the same probability space in such a way that (recall \eqref{coord})
\begin{equation}\label{conserve}
\forall t\ge 0,\ \bX^\wedge(t)\ge \bX^x(t).
\end{equation}
In the remainder of this proof $\bbP$ denotes the probability associated with the probability space on which this coupling is constructed.
We let $A_t$ denote the area comprised between the graphs of $\bX^x$ and $\bX^\wedge$:
$$ A_t := \sum_{k=1}^{N-1} \left(X_k^\wedge(t) - X^x_k(t)\right)\;,$$
and we aim at bounding its hitting time of $0$
$$\tau := \inf\{t\ge 0: A_t = 0\}.$$
As the ordering given by \eqref{conserve} implies that $\bX^\wedge(t)= \bX^x(t)$ for $t\ge \tau$, we have 
\begin{equation}
 \| P_t^x - P_t^\wedge \|_{TV}\le \bbP[\tau>t].
\end{equation}
Of course the distribution of $\tau$ depends tremendously on the coupling. For instance, the reader can check that for the coupling provided by Proposition \ref{grandgradient} which satisfies \eqref{conserve}, we have $\tau=\infty$ with probability one.
The coupling presented below is constructed with the aim of minimizing the merging time $\tau$.
The proof is split into three main steps:
\begin{enumerate}
\item The area passes below $N^{3/2}$ by time $t_{\delta/2}$ with large probability.
\item Within an additional time of order $N^2$, the area is very likely to decline from $N^{3/2}$ to  $N^{-1}$. 
\item 
The area goes from $N^{-1}$ down to $0$ within a time of order $\log N$ with large probability.
\end{enumerate}
The above steps clearly ensure that the event $\tau>t_\delta$ has a small probability, which will conclude the proof of Proposition \ref{Prop:Coalesce}.

The first step is a rather simple consequence of the fact that $\bbE[X_k^\wedge(t) - X^x_k(t)]$ is a solution of the discrete heat equation.

 The third step is a brute force argument inspired by Randall and Winkler~\cite{RW05}. More precisely, for this last step we build on the main idea of \cite{RW05}, but we improve it in a quantitative manner using estimates on the minimum gradient of $\bX^{\wedge}$.

The second step above is by far the most delicate one. The strategy is to introduce a sequence of intermediate thresholds between the values $N^{3/2}$ and $N^{-1}$, and then analyse the associated hitting times for the area process.  
Our control of these hitting times relies on diffusive estimates for the supermartingale $A_t$ and on estimates on the corresponding angle bracket process. One of the ingredients of the latter estimates is  a fine control of the  gradients of $\bX^{\wedge}$, which is in turn derived from a combination of equilibrium estimates and Proposition \ref{Prop:AbsCont}.

\subsection{The coupling and preliminary lemmas}\label{thecoupling}
We start by defining the coupling $\P$. For any $k\in\lint 1,N-1\rint$ and any configuration $x\in\Omega_N$, we define the ``interval of resampling of the $k$-th coordinate'' as $I(x,k) := [x_{k-1},x_{k+1}]$ and write $| I(x,k) | = \nabla x_k$ for its length, where the ``gradient" $  \nabla x_k$ is defined by 
\begin{equation}\label{Eq:Gradient}
\nabla x_k := x_{k+1}-x_{k-1}\;.
\end{equation}
For simplicity we sometimes  use the short-hand notation
$I^\wedge = I(\bX^\wedge(t_-),k)$ and $I^x = I(\bX^x(t_-),k)$.
 In our construction, we are going to try, at each resampling event, to couple $\bX^x_k$ and $\bX^{\wedge}_k$ with the maximal probability.
Letting $\Beta_{\alpha}(I)$ denoting the distribution with density given by (when $I=[a,b]$)
$$B_{\alpha}[a,b](x):= \frac{\gG(2\alpha) (x-a)^{\alpha-1}(b-x)^{\alpha-1}}{\gG(\alpha)^2 (b-a)^{2\alpha -1}} \ind_{[a,b]}(x)\;.$$
we set 
\begin{equation}\label{pnonexplicit}\begin{split}
p=p(t,k)&:= 1 - \| \Beta_{\alpha}( I^x)-\Beta_{\alpha}( I^{\wedge}) \|_{TV}\\
&=
\int_{I^\wedge \cap I^x}
\min( B_{\alpha}(I^x), B_{\alpha}( I^{\wedge}))(x) \dd x, 
\end{split}\end{equation}
and $q:=1-p$.
In the case $\alpha=1$ we have
\[p= \big| I^\wedge \cap I^x\big|/\max(\nabla X^\wedge_k(t_-), \nabla X^x_k(t_-))\] but there is no such simple expression for $p$ when $\alpha>1$; however we will be able to provide  good estimates for it (cf. Lemma \ref{Lemma:Bndq}).


Then, 
we define $\nu_1$, $\nu_2$, $\nu_3$ to be the probability measures with respective densities $\rho_1, \rho_2$, and $\rho_3$ given by
\begin{equation}\label{lesrhos}
 \begin{split}
  \rho_1&= q^{-1} \left[B_{\alpha}(I^x)-B_{\alpha}(I^{\wedge})\right]_+,\\
    \rho_2&= p^{-1}\min( B_{\alpha}(I^x), B_{\alpha}( I^{\wedge})) ,\\
    \rho_3&=q^{-1} \left[B_{\alpha}(I^{\wedge})-B_{\alpha}(I^x)\right]_+.
 \end{split}
\end{equation}
As a consequence of our assumption $\alpha\ge 1$ (which makes the functions $B_{\alpha}(I)$ unimodal), and the fact that by monotonicity both extremities of $I^{\wedge}$ are larger than their counterparts in $I^x$, the supports of $\rho_1$ and $\rho_3$ are intervals $I_1$ and $I_3$, the lower extremity of $I_3$ being larger or equal than the upper extremity of $I_1$. To see this it is sufficient to check that there exists at most one value $u$ such that the equation
$B_{\alpha}(I^x)(u)=B_{\alpha}(I^{\wedge})(u)>0$ has at most one solution. We refer to Figure \ref{fig:dens1} for the case $\alpha=1$ and to Figure \ref{fig:dens} for the general case of a unimodal density.  
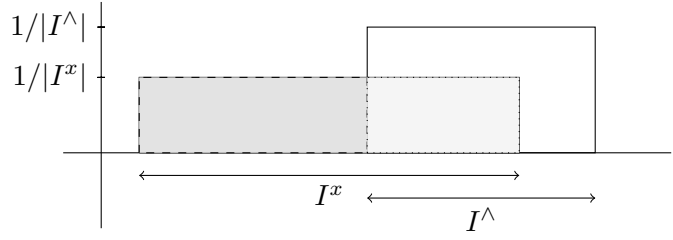
\begin{figure}
\centering
	\begin{tikzpicture}[scale=1.0]
	\draw[-,thin,color=black] (-1,0) -- (7,0);
	\draw[-,thin,color=black] (0,0) -- (0,1) -- (5,1) -- (5,0) -- (0,0);
	\draw[-,thin,color=black] (3,0) -- (3,5/3) -- (6,5/3) -- (6,0) -- (3,0);
	
	\draw[<->,color = black] (0,-0.3) -- (2.5,-0.3) node[below]{$I^x$} -- (5,-0.3);
	\draw[<->,color = black] (3,-0.6) -- (4.5,-0.6) node[below]{$I^\wedge$} -- (6,-0.6);
	
	\draw[dashed,fill=black!11] (0,0) rectangle (3,1);
	
	\draw[dotted,fill=black!4] (3,0) rectangle (5,1);
	
	\draw[-,thin,color=black] (-0.5,-1) -- (-0.5,1) -- (-0.45,1) -- (-0.55,1) node[left]{$1/|I^x|$} -- (-0.5,1) -- (-0.5,5/3) -- (-0.45,5/3) -- (-0.55,5/3) node[left]{$1/|I^\wedge|$} -- (-0.5,5/3) -- (-0.5,2);
	
	\end{tikzpicture}
	\caption{A plot of the densities $B_{\alpha}(I^x)$ and $B_{\alpha}(I^\wedge)$ when $\alpha=1$. The light, respectively dark, shaded region has area $p$, respectively $q$. }
\label{fig:dens1}
\end{figure}

\medskip

Our coupled dynamics is defined as follows. 
Each pair of  coordinates $(X^x_k,X^\wedge_k)$, $k=1,\dots,N-1$, is updated with rate one independently, and when an update occurs at time $t$ then 
\begin{itemize}
\item with probability $p$, the two new coordinates are set to the same value $X_k^x(t)=X_k^\wedge(t)$ drawn from the distribution $\nu_2$,
\item with probability $q$, the two new coordinates are sampled independently with respective distributions $\nu_1$ and $\nu_3$, and therefore satisfy $X_k^x(t)\le X_k^\wedge(t)$.
\end{itemize}
For convenience, we set
$$ \delta X_k(t) := X^\wedge_k(t) - X^x_k(t)\;,$$
 (in the remainder of the proof $\delta$ as a positive parameter  is not used anymore so that this should not yield confusion).
We also introduce the ``mean'' interfaces $\bar{\bX}^\wedge$ and $\bar{\bX}^x$ by setting
$$ \bar{X}^\wedge_k(t) := \frac{X^\wedge_{k-1}(t) + X^\wedge_{k+1}(t)}{2}\;,$$
and similarly for $\bar{\bX}^x$. Note that $\bar{X}^\wedge_k(t)$ is the midpoint of the interval of resampling of the $k$-th coordinate. We finally set
$$ \delta \bar{X}_k(t) := \bar{X}^\wedge_k(t) - \bar{X}^x_k(t)\;.$$

\medskip
We now collect a few facts on our coupling. We start by showing that the probability $q(t,k)$ can be fairly approximated by 
\begin{equation}
 Q(t,k):=\min\left(\frac{\delta {\bar X}_k(t_-)}{\max(\nabla X^\wedge_k(t_-),\nabla X^x_k(t_-))}, 1 \right)\;,
\end{equation}
and that, as soon as $q$ is close enough to $1$, the overlap between the resampling intervals represents a small fraction of the largest resampling interval. This second bound will be convenient in order to bound from below the derivative of the bracket of the area, see the proof of Proposition \ref{Prop:BoundBracket}.

\begin{lemma}\label{Lemma:Bndq}
For any $\alpha\ge 1$, there is a constant $C$ (that depends on $\alpha$) such that for all $t\ge 0$ and $k\in \lint 1,N-1\rint$
\begin{equation}\label{Eq:Bndq}
 \frac{1}{C} Q(t,k) \le q(t,k) \le  C  \,  Q(t,k).
\end{equation}
Furthermore, there exists a constant $c_1>0$ such that
\begin{equation}\label{Eq:Bndq2}
q(t,k) \ge 1- c_1 \quad \Rightarrow \quad \max(|I^x|, |I^{\wedge}|)\ge  2\big| I^{x}\cap I^{\wedge} \big|\;,
\end{equation}
(where the notation used is the one introduced below \eqref{Eq:Gradient}).
\end{lemma}
\begin{proof}
%

Note that \eqref{Eq:Bndq} is simply a result concerning the total variation between two $\Beta_{\alpha}$ variables defined on two intervals $I_1:=[l_1,r_1]$ and $I_2:=[l_2,r_2]$ such that $l_1\le l_2$ and $r_1\le r_2$. 
Recalling the definition \eqref{pnonexplicit} 
\begin{equation}
 q=\int_{0}^1 \left[ B_{\alpha}[l_1,r_1](u)- B_{\alpha}[l_2,r_2](u) \right]_+ \dd u,
\end{equation}
and \eqref{Eq:Bndq}  holds if we can prove  (for a different constant $C$) that 
\begin{equation}\label{I1I2}
\frac{1}{C} \min \left(\frac{\max(|l_2-l_1|, |r_2-r_1|)}{\max(r_1-l_1,r_2-l_2)},1 \right) \le  q \le C\,\frac{\max(|l_2-l_1|, |r_2-r_1|)}{\max(r_1-l_1,r_2-l_2)}.
\end{equation}
By symmetry, we can  assume that $I_1$ is the largest of the two intervals and  by scaling invariance we can assume without loss of generality that  $I_1=[0,1]$ and $I_2 =: [a,a+b]$, $a> 0 $, $b\in (1-a,1]$, in which case 
$$\max(|l_2-l_1|, |r_2-r_1|)/\max(r_1-l_1,r_2-l_2)=a.$$ We can further assume that $a<1$ as the result is trivially valid when $a\ge 1$.


For better readability of the proof we replace $\Beta_{\alpha}$ by a generic unimodal function $\rho$ which is positive on $(0,1)$, integrates to $1$ and is symmetric around $1/2$. We replace $\Beta_{\alpha}[a,a+b]$
by $\rho_{[a,a+b]}=b^{-1}\rho((\cdot -a)/b)$ with the convention that $\rho(u) = 0$ for $u\notin [0,1]$.

\medskip

\begin{figure}[htbp]
\centering
\leavevmode
\begin{center}
\resizebox{10cm}{!}{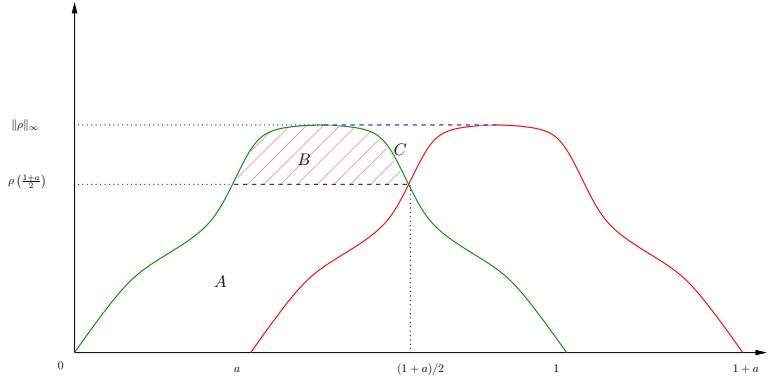}
 \end{center}
\caption{
We present a graphic proof of the inequalities \eqref{theup} and 
\eqref{thelow}. Using symmetry and unimodality of $\rho$ we see that the vertical stripe of constant width delimited by the two increasing portions of the graphs 
$\rho$ and $\rho_{[a,1+a]}$ can be divided into $3$ regions $A$, $B$ and $C$, the separation 
being given by the horizontal line $y=\rho(\tfrac{1+a}{2})$ and decreasing part of the graph of $\rho$. The integral we wish to estimate  $\int_{[0,1]} \left[ \rho(u)-  \rho_{[a,1+a]}(u)\right]_+ \dd u $
 corresponds to $|A|+|B|$ ($|\cdot|$ standing here for Lebesgue measure) while the upper and lower bounds $\|\rho \|_{\infty}a$ and $\rho\left(\tfrac{1+a}{2}\right)a$ correspond  respectively to $|A|+|B|+|C|$ and $|A|$. 
 }
\label{fig:dens}
\end{figure}
We first let the reader check that 
$\int_{[0,1]} \left[ \rho(u)-  \rho_{[a,a+b]}(u)\right]_+ 
\dd  u$ is increasing in $a$ and $b$  simply because the integrand displays the same monotonicities. 

\medskip

For \eqref{Eq:Bndq2}, we simply observe, using these monotonicities in $a$ and $b$, that if $a< 1/2$
then 
\begin{equation*}
 \int_{[0,1]} \left[ \rho(u)-  \rho_{[a,a+b]}(u)\right]_+ \dd u
 <  \int_{[0,1]} \left[ \rho(u)-  \rho_{[1/2,3/2]}(u)\right]_+ \dd u\;.
\end{equation*}
As a consequence, setting $c_1= 1-\int_{[0,1]} \left[ \rho(u)-  \rho_{[1/2,3/2]}(u)\right]_+ \dd u$, we deduce that if $q \ge 1-c_1$, we have $a \ge 1/2$ and $|I_1 \cap I_2 | \le 1/2$ thus yielding \eqref{Eq:Bndq2}.\\

\noindent To prove \eqref{Eq:Bndq}, we show that for every $a\in [0,1]$
\begin{equation}\label{Eq:TVBnd}
\frac{\rho(1/4) a}{4} \le \int_{[0,1]} \left[ \rho(u)-  \rho_{[a,a+b]}(u)\right]_+ \dd u \le \|\rho\|_{\infty}a.
\end{equation}
which allows to prove \eqref{I1I2} with $C= \max \left( \|\rho\|_{\infty},4/ \rho(1/4)\right)$.
By monotonicity it is sufficient to  check the upper-bound for $b=1$. Figure \ref{fig:dens} provides a graphical proof of the following inequality
\begin{equation}\label{theup}
\int_{[0,1]} \left[ \rho(u)-  \rho_{[a,1+a]}(u)\right]_+ \dd u \le \|\rho\|_{\infty}a.
\end{equation}
Now concerning the lower bound, using again  the graphical proof of Figure \ref{fig:dens}
(by symmetry of $\rho$ the two curves intersect at $(1+a)/2$), we have for $a\le 1/2$
\begin{equation}\label{thelow}
\int_{[0,1]} \left[ \rho(u)-  \rho_{[a,1+a]}(u)\right]_+ \dd u
\ge a \rho\left(\frac{1+a}{2}\right) \ge \rho(1/4)a\;,
\end{equation}
and thus the l.h.s.\ being increasing in $a$ we conclude that for all $a\in(0,1)$ (the factor $1/2$ is present so that the inequality is also valid for $a\in(1/2,1]$) 
$$\int_{[0,1]} \left[ \rho(u)-  \rho_{[a,1+a]}(u)\right]_+ \dd u
\ge \frac{\rho(1/4)a}{2} 
$$
Using symmetry and invariance by translation at the first line and the triangle inequality for the total variation distance at the second line, we have 
\begin{multline*}
\int_{[0,1]} \left[ \rho(u)-  \rho_{[a,1]}(u)\right]_+ \dd u
\\= \frac{1}{2}\left(\int_{[0,1]} \left[ \rho(u)-  \rho_{[a,1]}(u)\right]_+ \dd u +
\int_{[0,1]} \left[ \rho_{[a,1]}(u)-  \rho_{[a,1+a]}(u)\right]_+ \dd u  \right) \\ 
  \ge  \frac{1}{2}\int_{[0,1]} \left[ \rho(u)-  \rho_{[a,1+a]}(u)\right]_+ \dd u \ge  \frac{\rho(1/4)a}{4}.
\end{multline*}
This lower bound is also valid for $b\in(1-a,1)$ by monotonicity in $b$, thus completing the proof of \eqref{Eq:TVBnd}.
\end{proof}
 
 \begin{remark}\label{extension}
Let us observe for latter use that the upper bound in \eqref{I1I2} is also valid without the assumption $l_1\le l_2$, $r_1\le r_2$ (at the cost of taking $C$ twice as large).   
This can be deduced from the other case. Indeed,  assume without loss of generality that $[l_1,r_1]\subset [l_2,r_2]$ and set $I_3=[l_1,r_2]$. Then, using the previous bounds for the pairs $(I_1,I_3)$ and $(I_2,I_3)$,  
\begin{multline*}
 \| \Beta_{\alpha}( I_1)-\Beta_{\alpha}( I_2) \|_{TV}\\
 \le  \| \Beta_{\alpha}( I_1)-\Beta_{\alpha}( I_3) \|_{TV}+
  \| \Beta_{\alpha}( I_3)-\Beta_{\alpha}( I_2) \|_{TV}\\
  \le C \,\frac{|r_2-r_1|}{r_2-l_1}  +C\,\frac{|l_2-l_1|}{r_2-l_2}\le C\,\frac{\max(|l_2-l_1|, |r_2-r_1|)}{\max(r_1-l_1,r_2-l_2)}.
\end{multline*}

 \end{remark}
 Using the expression for the generator \eqref{genlinear} and the fact that our coupling preserves monotonicity the reader can check that for every $t\ge 0$
 \begin{equation*}
  \partial_{s}\bbE[ A_{t+s} \ | \ \cF_t ] \big|_{s=0}= \frac12\left(X^x_1(t)-X^{\wedge}_1(t)\right)+ \frac12\left(X^x_{N-1}(t)-X^{\wedge}_{N-1}(t)\right)\le 0,
 \end{equation*}
and hence that $A_t$ is a supermartingale for the filtration $(\cF_t)_{t\ge 0}$ defined by
$\cF_t:= \sigma( \bX^x_s, \bX^{\wedge}_s, s\le t)$. We write $\{\langle A_\cdot\rangle_t, t\ge 0\}$,   for the associated angle bracket process, namely the increasing predictable process that compensates the square of the martingale part of $A_t$.

 \begin{proposition}\label{Prop:BoundBracket}
There exists a constant $c>0$ such that for all $N$ large enough and for all $t\ge 0$, we have
$$ \partial_t \langle A_\cdot\rangle_t \ge c \sum_{k=1}^{N-1} \min\Big(\delta \bar{X}_k(t_{-}) \nabla X^\wedge_{k}(t_{-}) , \big(\nabla X^\wedge_{k}(t_{-})\big)^2 \Big)\;.$$
\end{proposition}
\begin{remark}
The proof will actually establish the same bound but with $\nabla X^\wedge_{k}(t)$ replaced by the maximum of the latter and $\nabla X^x_{k}(t)$.
\end{remark}
\begin{proof}
First of all, recalling the coupling defined in Section \ref{thecoupling}, we have
$$ \partial_t \langle A_\cdot \rangle_t = \sum_{k=1}^{N-1} p(t,k) (\delta X_k(t_-))^2 + q(t,k) J(t,k)\;,$$
where $J(t,k)$ is the mean square displacement corresponding
to an instantaneous uncoupled jump of $X^\wedge_k - X^x_k$ at time $t$. We are going to prove a lower bound for each term in the sum, and as before we omit from now on the dependence in $k$ and $t$.
Without loss of generality we assume that $\nabla X^\wedge \ge \nabla X^x$.

\medskip

We let $Z^x$ and $Z^\wedge$ denote the two independent variables with respective distribution $\nu_1$ and $\nu_3$ (whose densities are described in Equation \eqref{lesrhos}) that are used in the coupling.
We are going to prove first that 
\begin{equation}\label{lefirst}
 p (\delta X)^2 + qJ \ge q (\var( Z^x)+ \var (Z^{\wedge}))+  \frac{p}{q} (\delta \bar X)^2.
\end{equation}
This is achieved by computing explicitly $J$ (here $E$ denotes expectation for the pair of variables $(Z^x,Z^\wedge)$)
\begin{equation}
J = E[ (Z^{\wedge}-Z^{x}-\delta X)^2 ]
=\var( Z^x)+ \var (Z^{\wedge})+ (E[ Z^{\wedge}-Z^{x}]-\delta X)^2.
\end{equation}
Replacing $J$ by its value, observing that $\delta \bar X=  q E[ Z^{\wedge}-Z^{x}]$,
and taking the minimum over all possible values for $\delta X$
we obtain that the l.h.s.\ of \eqref{lefirst} is larger than
$$ q (\var( Z^x)+ \var (Z^{\wedge}))+  \min_{u\in \bbR} \left[ q(q^{-1}\delta \bar X -  u )^2 + p u^2 \right],$$
which is the desired result.

To conclude from \eqref{lefirst}, we consider $c_1$ from 
\eqref{Eq:Bndq2} in Lemma \ref{Lemma:Bndq}.
If $p \ge c_1$ then the r.h.s.\ of \eqref{lefirst}
is larger than $c_1(\delta \bar X)^2/q$ and we can conclude using the upper bound in \eqref{Eq:TVBnd}.

\noindent When $p\le c_1$ then we use  \eqref{Eq:Bndq2}
which ensures that
%
$$|I^\wedge \setminus I^x|\ge |I^\wedge|/2\;,$$
so that with probability at least $1/2$, $Z^\wedge$ coincides with a $\Beta_{\alpha}(I^\wedge)$ r.v.~conditioned on being larger than its median $\bar{X}^\wedge$. Since the variance of the latter conditional law is of order $|I^\wedge|^2 = (\nabla X^\wedge)^2$, so that for some adequate choice of $c>0$ we have
we have
\begin{equation}\label{Eq:Varineq}
 q\var (Z^{\wedge})\ge c q(\nabla X^{\wedge})^2\ge c(1-c_1)( \nabla X^{\wedge})^2,
\end{equation}
which concludes the proof.
\end{proof}

\subsection{Successive hitting times}

To prove Proposition \ref{Prop:Coalesce}, our argument is to show that 
by time $t_{\delta/2}+N^2$, the area has become very small (smaller than $N^{-1}$) and then to use some brute force argument to show that 
$\tau$ cannot be much larger.

\medskip

Our strategy to control the decay of the area requires several steps and 
we introduce the successive hitting times by the area of a sequence of well-chosen thresholds. For $\eta > 0$ small, we define
$$ \cT_i := \inf\{t\ge t_{\delta/2}: A_t \le N^{\frac  32 - i\eta}\}\;,\quad i\ge 1\;.$$
We first show that with large probability $\cT_2$ is equal to $t_{\delta/2}$. Then setting $L:= \min\{i\ge 1: \frac32 - i\eta < -1\}$, we show that  
the increments $\cT_i-\cT_{i-1}$ are small for all $i\le L$. The argument to control the increment $\cT_i - \cT_{i-1}$ differs for different ranges of $i$ so that it is practical for us to introduce the following intermediate thresholds
$$I:= \max\{i\ge 1: \frac32 - i\eta > 1\}\;,\qquad K := \max\{i\ge 1: \frac32 - i\eta > 1/4\}\;.$$
We now introduce the following events which allow us to split our proof in four parts
\begin{equation}
 \begin{split}
  \cB_{(1)}^N &:= \{\cT_2=t_{\gd/2} \}, \\  \cB_{(2)}^N &:=\{\forall i \in \lint 3,I \rint, \cT_i-\cT_{i-1}\le 2^{-i}N^2\},\\
  \cB_{(3)}^N &:=\{\forall i \in \lint I+1,K \rint, \cT_i-\cT_{i-1}\le 2^{-i}N^2\},\\
    \cB_{(4)}^N &:=\{\forall i \in \lint K+1,L \rint, \cT_i-\cT_{i-1}\le 2^{-i}N^2\}.
 \end{split}
\end{equation}
We prove in the forthcoming Sections \ref{sec:A1}, \ref{sec:A2}, \ref{sec:A3}, and \ref{sec:A4} respectively that each of these four events holds with probability tending to 1. This implies in particular that
$$\lim_{N\to \infty} \bbP\left[\cT_L \le t_{\delta/2}+N^2\right]=1\;.$$
In Section \ref{sec:conclude} we use this last statement to conclude the proof of Proposition \ref{Prop:Coalesce}.
The remaining Sections \ref{sec:tech1}-\ref{sec:tech2} are dedicated to the introduction of technical material which is used throughout Sections \ref{sec:A2}-\ref{sec:conclude}.

\subsection{Initial contraction of $A_t$ and control of $\cT_2$}\label{sec:A1}

The probability of $\cB_{(1)}^N$ can be controlled  using Markov's inequality for the non-negative random variable $A_t$. Indeed $\E[A_t]$ has an explicit expression in terms of the discrete heat equation.  This argument does not exploit our specific coupling $\P$.

\begin{lemma}\label{Lemma:Contract}
For any $\eta >0$ we have as $N\to\infty$
$$ \lim_{N\to \infty} \P(A_{t_{5\eta}} > N^{\frac32 - 2\eta}) =0.$$
In particular, if $\eta\le \delta/10$ then $\lim_{N\to \infty}\P( \cB_{(1)}^{N})=1$.  
\end{lemma}
\noindent
\begin{proof}
Let us set $a(t,k) := \E[X^\wedge_k(t) - X_k^x(t)]$. From the expression \eqref{genlinear}, we deduce that
$$ \partial_t  a(t,k)= \frac12 \Delta a(t,k)\;,\quad k\in\lint 1,N-1\rint\;.$$
In addition we have $a(t,0) = a(t,N) = 0$. Expanding $k\mapsto a(t,k)$ on the orthonormal basis of the discrete Laplacian as in \eqref{Eq:XiFourier} and estimating all the corresponding eigenvalues by the main eigenvalue, it follows that
$$ \sum_{k=0}^N a(t,k)^2 \le e^{-2\gap_N t} \sum_{k=0}^N a(0,k)^2\;.$$
By Cauchy-Schwarz's inequality,
\begin{equation}\label{usingCS}
\E[A_t]^2 \le N \sum_{k=0}^N a(t,k)^2 \le N e^{-2\gap_N t} \sum_{k=0}^N a(0,k)^2 \le N^4 e^{-2\gap_N t}\;.
\end{equation}
Markov's inequality then yields the asserted result. 
\end{proof}

\subsection{Technical preliminaries to control the probability of $\cB_{(2)}^N$, $\cB_{(3)}^N$,  $\cB_{(4)}^N$}\label{sec:tech1}

Before going into the specifics of each case let us introduce the common framework which allows us to control $\cT_i-\cT_{i-1}$ for $i\in \lint 3,L\rint$.
Our main idea is to exploit the fact that $(A_t)_{t\ge 0}$ is a supermartingale for which we have a reasonable control on the jumps.
For such processes the hitting time can be estimated if one can control the angle bracket of the martingale, as shown in the following result from~\cite{LabLacWASEP}.

\begin{proposition}\label{prop:solskjaer}
Let $(M_t)_{t\ge 0}$ be a pure-jump supermartingale with bounded jump rate and jump amplitude.
Given $a \in\bbR$ and $b\le a$, set $$\tau_{b}:=\inf\{ t\ge 0 \ : \ M_t\le b\}.$$
If the amplitude of the jumps of $(M_t)_{t\ge 0}$ is bounded above by $d$, then
we have for any $v\ge 4d^2$
\begin{equation}
 \bbP[  \langle M_\cdot \rangle_{\tau_b}\ge v  \ | \ M_0\le a ]\le 8  (a-b)v^{-1/2},
\end{equation}
where  $\langle M_\cdot \rangle$ denotes the bracket of $(M_t)_{t\ge 0}$
(the predictable processes which compensates the square of the martingale part of $M$).
\end{proposition}
\begin{proof}
The only difference with the statement of~\cite[Proposition 29]{LabLacWASEP} is that $d$ is not necessarily below $a-b$. However, a careful inspection of the proof therein shows that the present statement holds (note that the proof therein relies on~\cite[Lemma 30]{LabLacWASEP} and this result does not need any modification to cover our setting).
\end{proof}

We apply the previous proposition to the supermartingale $M_t=A_t$. 
Our idea is to combine the above result with 
estimates on the increments of the bracket of $A$, $\Delta_i \langle A\rangle := \langle A_\cdot\rangle_{\cT_i}-\langle A_\cdot\rangle_{\cT_{i-1}}$ 
in order to obtain upper-tail probability for the increments of the hitting times $\cT_i-\cT_{i-1}$.
The control of $\Delta_i \langle A\rangle$ as a function of $\cT_{i}-
\cT_{i-1}$ is technically involved.
Our general strategy is to restrict ourselves to an event of large probability for which we have (for some adequate constant $C(i,N)$)  
$$ \Delta_i \langle A\rangle \ge C(i,N) (\cT_{i}-
\cT_{i-1}).$$
The specific events that we have to consider is introduced in the following section.

\subsection{Restriction to the right set of events}\label{sec:tech2}

Our convenient event is the intersection of two events $\cA^{(N)}_1$ and $\cA^{(N)}_2$ that we now introduce (for these events and others we make the dependence in $N$ appear only when necessary). Regarding $\cA_1$, we only impose that the increments of the higher interface are not too large ``at all times'':
\begin{equation}
 \cA_1:=\big\{ \forall t\in [t_{\delta/2},t_{\delta/2}+N^2]: \,  \max_{k\in \lint 1,N-1 \rint}\eta_k^{\wedge}(t)\le 10 \log N \big\}. 
\end{equation}
That the probability of $\cA_1$ goes to $1$ will follow from the fact that the higher interface is close to equilibrium by Proposition \ref{Prop:AbsCont} and from simple estimates under the invariant measure. To define $\cA_2$, we introduce the following events that impose some restrictions on the interfaces at a given time $t$. The events $\cC_1$ and $\cC_2$ require respectively the gradients of the higher interface to be not too small, and the distance between the two interfaces to be not too large:
\begin{equation}\begin{split}
\cC_1(t)&:=\big\{ \min_{k\in \lint 1,N-1 \rint} \nabla X^{\wedge}_k(t) \ge  (N \log N)^{-1/2} \big\},\\
\cC_2(t)&:=\big\{ \max_{k\in \lint 1,N-1 \rint} |X^{\wedge}_k-X^{x}_k| \le  \sqrt{N} \log N \big\},
                \end{split}
\end{equation}
Given $x \in \Omega_N$, we let $(a_{i}(x))_{i=1}^{N-1}$ be
the increasingly ordered sequence of the values $(\nabla x_k)_{k=1}^{N-1}$. Then we set
\begin{equation}
 \cC_3(t):=\left\{ \forall i\in \lint 1,  N \rint :\quad
 (a_{i}(X^\wedge(t)))^2\ge \frac{i}{N\log N} \right\}\;.
\end{equation}
Finally we define
\begin{equation}
 \cC_4(t):=\left\{ \forall i\in \lint 0,N- (\log N)^2\rint \ :\quad
 \sum_{k=i+1}^{i+(\log N)^2} (\nabla X^{\wedge}_k(t))^2\ge \frac{(\log N)^2}{100} \right\}\;.
\end{equation}
The event $\cA_2$ then requires that for a large proportion of the interval of time $[t_{\delta/2},t_{\delta/2}+N^2]$, the four events $\cC_i(t)$ are satisfied:
\begin{equation}
 \cA_2:=\left\{ \int_{t_{\delta/2}}^{t_{\delta/2}+N^2}  \ind_{\cC_1(s)\cap\cC_2(s)\cap\cC_3(s)\cap \cC_4(s)}   \dd s \ge (1-2^{-(L+1)}) N^{2} \right\}. 
\end{equation}
Note that the probability of $\cA_2$ still goes to $1$ if $1-2^{-(L+1)}$ is replaced by any factor $1-c \in [0,1)$. However, for latter use we need this factor to be larger than $1-2^{-L}$, and this explains our particular choice $1-2^{-(L+1)}$.

\begin{proposition}\label{Prop:A1A2}
We have $\lim_{N\to \infty} \bbP[\cA^{(N)}_1\cap \cA^{(N)}_2]=1$.
\end{proposition}
\begin{proof}
Recall that $\pi_{N,\alpha}$ is the invariant measure of our dynamics, and let $\nu_N$ be the law on $\Omega_N^+$ under which the $\eta_k$'s are independent $\gG(\alpha,\alpha)$ r.v. Without further mention, we will apply Lemma \ref{Lemma:AbsCont} that allows us to bound some functionals under $\pi_{N,\alpha}$ by the same functionals under $\nu_N$.

We start with the event $\cA_1$. We have (recall that $\alpha \ge 1$)
\begin{equation}\label{Eq:BoundGrad}
\pi_{N,\alpha}(\exists k: \eta_k \ge 10 \log N) \le C \sum_{k=1}^N \nu_N(\eta_k \ge 10 \log N) \le N^{-8}\;.
\end{equation}
By Proposition \ref{Prop:AbsCont}, it follows that 
$$\lim_{N\to \infty} \sup_{t\ge t_{N,\delta/2}} 
\| P_{t}^\wedge - \pi_{N,\alpha} \|_{TV} = 0\;.$$
Therefore, it suffices to work with the process starting from the stationary measure $\pi_{N,\alpha}$: we denote by $\bP$ the law of such a process. Let us subdivide the interval $[t_{\delta/2},t_{\delta/2}+N^2]$ into disjoint intervals $[t_i,t_{i+1}]$ of size $N^{-5}$: note that there are of order $N^7$ such intervals. Then, a standard argument on independent Poisson clocks ensures that the probability that on each interval $[t_i,t_{i+1}]$ there is no more than $1$ resampling event is larger than $1-CN^{-1}$ for some constant $C>0$, hence on that event, at any time $s\in [t_{\delta/2},t_{\delta/2}+N^2]$, $\eta_k(s)$ is equal to some $\eta_k(t_i)$ for some $i$. Moreover, a simple union bound combined with \eqref{Eq:BoundGrad} shows that the probability that for all $t_i$ and all $k$, $\eta_k(t_i) < 10 \log N$ goes to $1$. Therefore,
$$ \lim_{N\to \infty }\bP(\exists t\in [t_{\delta/2},t_{\delta/2}+N^2], \exists k\in \lint 1,N\rint: \eta_k(t) \ge 10 \log N) = 0\;.$$

We turn to the event $\cA_2$. By Markov's inequality, it suffices to show that for every $i\in\lint 1,4\rint$ 

$$ \lim_{N\to \infty}\sup_{t\ge t_{\delta,N}}\bbP(\cC^{(N)}_i(t)^\cc)=0.$$
To handle the events $\cC_1$, $\cC_3$ and $\cC_4$, Proposition \ref{Prop:AbsCont} ensures that one can work under the stationary measure $\pi_{N,\alpha}$.
We are going to make extensive use of the following tail estimates for the $\gG(2\alpha,\alpha)$ distribution:
\begin{equation}\label{tail}
 \nu_N( \nabla x_k \le t )=\nu_N(\eta_1+\eta_2\le t) \le C' t^{2\alpha} \le C t^2, \quad \forall t\in [0,1]\;.
\end{equation}
By Lemma \ref{Lemma:AbsCont}, the bound is also valid under $\pi_{N,\alpha}$ (with a different constant). To control the probability of $\cC^\cc_1$ it is sufficient to observe that by union bound
and exchangeability of the increments
\begin{equation}
\pi_{N,\alpha}(\cC^{\cc}_1)\le N\pi_{N,\alpha}\Big(\nabla x_1 < \frac1{\sqrt{N \log N}}\Big) \le C (\log N)^{-1}\;.
\end{equation}

We turn to $\cC_4$. By union bound and exchangeability again
we have 
\begin{align*}
 \pi_{N,\alpha}(\cC^{\cc}_4) &\le N \pi_{N,\alpha} \left( \sum_{k=1}^{(\log N)^2} (\nabla x_k)^2\le \frac{(\log N)^2}{100} \right)\\ 
 &\le C N \nu_N\left( \sum_{k=1}^{\frac12 (\log N)^2} (\nabla x_{2k})^2\le \frac{(\log N)^2}{100} \right)\\
 &\le CN \Big(\nu_N\big[e^{-\gl (\nabla x_2)^2}\big] \, e^{\frac{\gl}{50}}\Big)^{\frac{(\log N)^2}{2}}\;,
\end{align*}
for all $\gl \ge 0$. Since $\nu_N[(\nabla x_{2k})^2] = (4+2\alpha^{-1}) > 1$, it is simple to check that for small enough $\gl$, we have $\E[e^{-\gl (\nabla x_2)^2}] \le e^{-\gl}$ which suffices to conclude.

%

We now consider $\cC_3$. We let $b_i(x)$ denote the set of increasingly ordered values of $(\nabla x_{2k})_{k=1}^{N/2}$. We are going to show that $\pi_{N,\alpha}(\cC'^{\cc}_3)$ tends to zero where
$$\cC'_3:=\left\{\forall   i\le N/4, \ b_i(x)^2 \ge \frac{4i}{N\log N}\right\}.$$
The same bound concerning the odd gradients $(\nabla x_{2k-1})_{k=1}^{N/2}$ is then sufficient to conclude. 
By union bound and exchangeability of the variables $(\nabla x_{2k})_{k=1}^{N/2}$, we have
\begin{equation*}
\pi_{N,\alpha}((\cC'_3)^{\cc})\le \sum_{i=1}^{N/4} \binom{N}{i} \pi_{N,\alpha}
\left( \forall j\in \lint 1,i \rint,\   (\nabla x_{2j})^2\le \frac{4i}{N\log N} \right).
\end{equation*}
Using Lemma \ref{Lemma:AbsCont} and \eqref{tail}
we obtain for some positive constant $C$
\begin{equation}
\pi_{N,\alpha}((\cC'_3)^{\cc})\le  \sum_{i=1}^{N/4} \frac{N^i}{i!} \left(\frac{Ci}{N\log N}\right)^i\le \frac{C'}{\log N}.
\end{equation}


Regarding $\cC_2$, we first show that for $N$ sufficiently large
\begin{equation}\label{Eq:InvGamma}
\pi_{N,\alpha}\left(\max_{k\in \lint 1,N-1 \rint} |x_k - k| > \sqrt N \log N /2 \right) \le N^{-1}\;.
\end{equation}
By symmetry, we can restrict to $k\in \lint 0,N/2\rint$ and by Lemma \ref{Lemma:AbsCont}, it suffices to prove this bound under $\nu_N$. For every $k\in \lint 0,N/2\rint$ the law of $x_k$ under $\nu_N$ is $\gG(k\alpha,\alpha)$. Using a Chernoff bound for all $N$ large enough we obtain
$$ \nu_N\left(|x_k - k| > \sqrt N \log N/2\right) \le e^{-10\log N}\;.$$
(Note that a finer upper bound would be $e^{-c (\log N)^2}$ for some small enough $c>0$).
Consequently, \eqref{Eq:InvGamma} follows. By Proposition \ref{Prop:AbsCont}, we deduce that 
\begin{equation}\label{wedgefluc}
\lim_{N\to \infty}\sup_{t\ge t_{\delta/2}}\bbP\left(\max_{k\in \lint 1,N-1 \rint} |X_k^\wedge(t) - k| > \sqrt N \log N/2\right)=0 \;.
\end{equation}
By symmetry the same is valid for $X_k^\vee(t)$.
By Proposition \ref{grandgradient}  $X_k^x(t)$ is stochastically dominated by $X_k^\wedge(t)$ and stochastically dominates $X_k^\vee(t)$.
This implies that 
\begin{equation}\label{xfluc}
\lim_{N\to \infty}\sup_{x\in \gO_N}\sup_{t\ge t_{\delta/2}}\bbP\left(\max_{k\in \lint 1,N-1 \rint} |X_k^x(t) - k| > \sqrt N \log N/2\right)=0 \;.
\end{equation}
and the result is obtained by observing that $\cC_2(t)^{\cc}$ is contained in the union of the two events in \eqref{wedgefluc} and \eqref{xfluc}.
\end{proof}

\subsection{Controlling the $\cT_i$ increments for $i\in \lint 3,I\rint$}\label{sec:A2}

We are now ready to prove 
\begin{lemma}\label{lem:B2}
 The event $\cB_{(2)}^N$ satisfies
 \begin{equation}
  \lim_{N\to \infty} \bbP[\cB_{(2)}^N]=1.
 \end{equation}

\end{lemma}
Recall that for any $t\in [t_{\delta/2},\cT_I)$, the area satisfies $A_t > N^{1+\kappa}$ for some $\kappa > 0$. Our proof relies then on the following observation.

\begin{lemma}\label{Lemma:LowBd1}
For all $t\in [t_{\delta/2},t_{\delta/2}+N^2]$,
on the event $\cA_1\cap\cC_2(t)\cap\cC_4(t)\cap\{ A_t\ge N (\log N)^4 \}$ we have 
\begin{equation}
\partial \langle A_{\cdot} \rangle_t \ge \frac{A_t}{\sqrt{N} \log N}. 
\end{equation}
\end{lemma}

\begin{proof}[Proof of Lemma \ref{lem:B2}]
Proposition \ref{prop:solskjaer} applied to the supermartingale $ (A_{\cT_{i-1} +s})_{s\ge 0}$ (whose maximal jump size is $N$) with $v=N^{3-2(i-1)\eta} \log N$, yields that $\lim_{N\to \infty}\P(\cA^{(N)}_3) =1$ where 
$$ \cA_3:= \bigcap_{i=3}^{I} \big\{\Delta_i \langle A\rangle < N^{3-2(i-1)\eta} \log N\big\}\;.$$
To conclude we only need to show that   
$$\cB_{(2)}^N \supset \cA_1\cap \cA_2 \cap \cA_3\cap \cB_{(1)}^N.$$
We proceed by contradiction. On the event $\cA_1\cap \cA_2 \cap \cA_3$ consider the smallest integer $j$ in $\lint 3,I\rint$ such that $\cT_j-\cT_{j-1} > 2^{-j}N^2$. Applying Lemma \ref{Lemma:LowBd1} 
on the interval
$[\cT_{j-1},\cT_{j-1} +2^{-j}N^2]\subset [t_{\delta/2},t_{\delta/2}+N^2]$,
on which $A_t\ge N^{3/2-j\eta}$
we obtain 
\begin{multline}
\Delta_j \langle A\rangle \; \ge  \langle A\rangle_{\cT_{j-1}+2^{-j}N^2 }-\langle A\rangle_{\cT_{j-1} } \\
\ge  \frac{N^{3/2-j\eta}}{\sqrt{N} \log N} \int_{\cT_{j-1}}^{\cT_{j-1}+2^{-j}N^2} \ind_{\cC_2(s)\cap\cC_4(s)}\dd s \ge 
\frac{2^{-(j+1)}N^2 N^{3/2-j\eta}}{\sqrt{N} \log N},
\end{multline}
where the last inequality uses the definition of $\cA_2$ to assert that 
$$\int_{\cT_{j-1}}^{\cT_{j-1}+2^{-j}N^2} \ind_{\cC_4(s)\cap \cC_2(s)}\dd s\ge 
(2^{-j}-2^{-(L+1)})N^2.$$

\end{proof}

\begin{proof}[Proof of Lemma \ref{Lemma:LowBd1}]
Since $A_t \ge N (\log N)^4$, we have for $N$ large enough
$$ \sum_{k=1 }^{N-1} \delta X_k(t)\ind_{\{\delta X_k(t) \le 20(\log N)^3 \}} \le 20 N (\log N)^3 \le \frac12 A_t\;.$$
Since we work on $\cC_2(t)$, we get
\begin{align}\label{semis}
\frac12 A_t &\le \big(\max_{k\in \lint 1,N-1\rint} \delta X_k(t)\big)\; \#\{k: \delta X_k(t) > 20(\log N)^3\}\\
&\le \sqrt{N} \log N\;\#\{k: \delta X_k(t) > 20(\log N)^3\}\;.
\end{align}
The bound on the gradients given by $\cA_1$ ensures that 
$$\delta X_{k-1}(t) > \delta X_k(t) - 10\log N$$ so that for $N$ large enough we get
$$ \left\{ \delta X_k(t) > (\log N)^3 \right\} \Rightarrow \left\{\delta \bar{X}_k(t) > \frac{(\log N)^3}{3}  \right\}\;,$$
and
$$ \left\{\delta X_\ell(t) > 20(\log N)^3 \right\}\Rightarrow \Big\{\forall k\in \lint \ell- (\log N)^2,\ell \rint,\; \delta X_k(t) > (\log N)^3\Big\}\;.$$
Consequently by Proposition \ref{Prop:BoundBracket} on the event $\cA_1$
\begin{align*}
\partial_t \langle A_\cdot\rangle_t &\ge c \sum_{k=1 }^{N-1}  \big(\nabla X^\wedge_{k}(t)\big)^2\ind_{\{\delta X_k(t) > (\log N)^3\}}\\
&\ge c\sum_{k=1}^{N-1}  \big(\nabla X^\wedge_{k}(t)\big)^2 \ind _{\{ \exists \ell \in \lint k,k+(\log N)^2 \rint:\, \delta X_\ell(t) > 20(\log N)^3\}}\;.
\end{align*}
The set over which the sum is taken on the r.h.s.~can be decomposed into connected components of size at least $(\log N)^2$. On $\cC_4(t)$, we thus deduce that there exists $c'>0$ such that
\begin{align*}
\partial_t \langle A_\cdot\rangle_t &\ge c' \;\#\{k: \exists \ell \in \lint k,k+(\log N)^2 \rint, \delta X_\ell(t) > 20(\log N)^3\} \\
&\ge c' \;\#\{\ell: \delta X_\ell(t) > 20(\log N)^3\}\\
&\ge \frac{c'}{2} \frac{A_t}{\sqrt N \log N}\;,
\end{align*}
where in the last line we simply used \eqref{semis}.
\end{proof}

\subsection{Controlling the $\cT_i$ increments for $i\in \lint I+1,K\rint$}\label{sec:A3}

We can prove now
\begin{lemma}\label{lem:B3}
 We have 
 \begin{equation}
  \lim_{N\to \infty} \bbP[\cB_{(3)}^N]=1.
 \end{equation}
 \end{lemma}
We follow essentially the same line of proof as in the previous section, but using a different inequality for the bracket derivative. We also need an extra trick to control the maximal amplitude of jumps.
 Recall that for every $t\in [\cT_I,\cT_K)$, we have $A_t > 2N^{1/4}$.

\begin{lemma}\label{Lemma:LowBd2}
Fix $\gep > 0$. On the event $\cA_1\cap\cC_3(t)\cap \cC_4(t)\cap \{A_t > 2N^{\gep}\}$, we have for all $t\in [t_{\delta/2},t_{\delta/2}+N^2]$
$$ \partial_t \langle A_\cdot \rangle_t \ge N^{\frac{\gep}{2}} \wedge \frac{A_t^2}{N^{1 + 3\gep}}\;.$$
\end{lemma}

\begin{proof}[Proof of Lemma \ref{lem:B3}]

In order to diminish the restriction on $v$ imposed in Proposition \ref{prop:solskjaer}
we introduce the following stopping times with the objective of considering a process with smaller jump amplitude
$$\cR_i:= \inf\{t\ge \cT_{i-1}: \max_{k\in \lint 1, N-1\rint} \nabla X_k^\wedge(t) > 10 \log N \mbox{ or } A_t > N^{\frac32 - (i-2)\eta}\}\;.$$
We consider the supermartingale $(A_{(\cT_{i-1} +t)\wedge \cR_i})_{t\ge 0}$. For every $i\le L$, its maximal jump size up to time $\cT_i-\cT_{i-1}$ is bounded by
$$ 10 \log N + N^{\frac32 - (i-2)\eta}.$$
Indeed, the maximal variation of the area due to an update of site $k$ at time $t$ is bounded above by
$$ X_{k+1}^\wedge(t_-) - X_{k-1}^x(t_-) \le \delta X_{k-1}(t_-) + \nabla X_k^\wedge(t_-) \le A_{t_-} + 10 \log N\;.$$
We consider now the event
\begin{equation}\begin{split}
                 \cA^{(N)}_{4,1}&:= \{ \forall i \in \lint I+1,L \rint, \forall t\ge \cT_{i-1}, A_t \le  N^{\frac32 - (i-2)\eta}\}\\
                  \cA^{(N)}_{4,2}&:=\{  \forall i \in \lint I+1,L \rint, \ \langle A\rangle_{\cT_{i}\wedge \cR_i} - \langle A\rangle_{\cT_{i-1}\wedge \cR_i}  \le v_i \},\\
\cA^{(N)}_{4}&:= \cA^{(N)}_{4,1}\cap \cA^{(N)}_{4,2}.
                \end{split}
\end{equation}
with $v_i:=4 (10 \log N + N^{\frac32 - (i-2)\eta})^2$.
The standard union bound and a supermartingale version of Doob's maximal inequality (sometimes refered to as Ville's inequality) show that 
$\bbP[(\cA^{(N)}_{4,1})^{\cc}]\le 3(L-I) N^{-\eta}$, and hence converges to $0$.
Now, applying Proposition \ref{prop:solskjaer} to $(A_{(\cT_{i-1} +t)\wedge \cR_i})_{t\ge 0}$ with $v=v_i$ yields $\lim_{N\to \infty}\P(\cA^{(N)}_{4,2})=1$.

\medskip

\noindent Our final step is to prove 
$$\cB_{(3)}\supset \cA_1\cap \cA_2 \cap \cA_4 \cap \cB_{(1)} \cap \cB_{(2)}\;.$$
We place ourselves on the event $\cA_1\cap \cA_2 \cap \cA_4 \cap \cB_{(2)}$ and we proceed by contradiction.  Consider the smallest element $j$ of $\lint 1,K\rint$ such that 
$\cT_j-\cT_{j-1} \ge 2^{-j} N^2$ (note that as we are on $\cB_{(2)}^N$ we must have $j\ge I+1$).
Because we are on the event $\cA_1 \cap \cA_{4,1}$, the reader can check that we must have 
$$\cR_j \ge t_{\delta/2}+N^2 \ge \cT_{j-1} + 2^{-j} N^2.$$
Hence by Lemma \ref{Lemma:LowBd2} , $\cA_{4,2}$ and $\cA_2$ imply that 
\begin{multline}
 v_j \ge \langle A\rangle_{\cT_{j-1}+2^{-j} N^2} - \langle A\rangle_{\cT_{j-1}} \\ \ge  \Big(N^{\frac18} \wedge N^{\frac{5}{4}-2j\eta}\Big) \int^{\cT_{j-1}+2^{-j} N^2}_{\cT_{j-1}}\ind_{\cC_3(s)\cap \cC_4(s)}\dd s \ge 2^{-j-1} \Big(N^{\frac{17}{8}} \wedge N^{\frac{13}{4}-2j\eta}\Big),
 \end{multline}
where for the inequality we proceed as in the proof of Lemma \ref{lem:B2}
to show that the integral is larger than $2^{-j-1}N^2$. This yields the desired contradiction.
%
\end{proof}

We are left with proving Lemma \ref{Lemma:LowBd2}. The first step is the following technical estimate.
\begin{lemma}\label{Lemma:Baibi}
 Given $B>0$, let $(a_i)_{i=1}^n$  be an increasing sequence of positive real numbers with 
 $\max a_i\le B$, and $(b_i)_{i=1}^n$  an arbitrary sequence in $\bbR_+$ with $\max b_i\le B$, 
 and let us set $\sigma:=\sum_{i=1}^n b_i$, and $K:=\lfloor \sigma/B \rfloor$.
 Then we have 
 \begin{equation}
  \sum_{i=1}^n a_i (a_i \wedge b_i)\ge  \sum_{i=1}^K a^2_i
 \end{equation}
If $K=0$ we have 
\begin{equation}\label{subopt}
 \sum_{i=1}^n a_i (a_i \wedge b_i)\ge  a_1(a_1\wedge \sigma).
\end{equation}
\end{lemma}
\begin{proof}
 Let us start with the case when $\sigma\ge B$. 
 Using that $a_i$ and $b_i$ are bounded by $B$ we have
 $$ \sum_{i=1}^n a_i (a_i \wedge b_i)\ge \sum_{i=1}^n \frac{a^2_i b_i}{B}.$$
 Now the reader can check that if $\sigma$ is fixed,
 the r.h.s.~is minimized when $b_i=B$ for $i\le K$, $b_{K+1}=\sigma-BK$, and $b_i=0$ for $i>K+1$.
 When  $\sigma< B$, it is sufficient to check \eqref{subopt} when $\sigma\le a_1$ by monotonicity, and in that case we have  
 $$\sum_{i=1}^n a_i (a_i \wedge b_i)=\sum_{i=1}^n a_i b_i \ge a_1 \sum_{i=1}^n b_i.$$
\end{proof}
\begin{proof}[Proof of Lemma \ref{Lemma:LowBd2}]
First assume that $\max_k |X^{\wedge}_k(t)-X^{x}_k(t)| > N^{\gep}$. As in the proof of Lemma \ref{Lemma:LowBd1}, the bound on the gradients given by $\cA_1$ ensures that
$$ \delta X_\ell(t) > N^\gep \Rightarrow \Big(\forall k\in \lint \ell- N^{\frac{2\gep}{3}},\ell \rint,\; \delta X_k(t) > \frac{N^\gep}{2}\Big)\;.$$
Consequently, if we let $\ell$ be an integer for which $\delta X_\ell(t) > N^\gep$ we get the bound
\begin{align*}
\partial_t \langle A_\cdot\rangle_t &\ge c \sum_{k\in \lint \ell- N^{\frac{2\gep}{3}},\ell \rint}  \big(\nabla X^\wedge_{k}(t)\big)^2\;.
\end{align*}
By $\cC_4(t)$, we thus deduce that 
\begin{align*}
\partial_t \langle A_\cdot\rangle_t &\ge N^{\frac{\gep}{2}} \;.
\end{align*}
Next, assume that $\max_k |X^{\wedge}_k(t)-X^{x}_k(t)| \le N^{\gep}$. Set $B=N^{\gep}$. We apply Lemma \ref{Lemma:Baibi} with $(a_i)_{i=1}^{N}$ being the ordered sequence of the $(\nabla X_k^\wedge(t))_{k=1}^N$ and $(b_i)_{i=1}^N$ being the corresponding $\delta \bar X_k(t)$'s. Since $\sum_k \delta \bar{X}_k(t) \ge A_t /2$ and since $A_t\ge 2 N^\gep$, the lemma combined with $\cC_3(t)$ and Proposition \ref{Prop:BoundBracket} gives the following lower bound
$$ \partial_t \langle A_\cdot\rangle \ge c \sum_{i=1}^{\lfloor (A_t N^{-\gep})/2 \rfloor} a_i^2 \ge \frac {c' (A_t N^{-\gep})^2} {N\log N} \ge \frac{A_t^2}{N^{1+3\gep}}\;.$$
\end{proof}

\subsection{Controlling the $\cT_i$ increments for $i\in \lint K+1,L\rint$} \label{sec:A4}

Following our plan we now prove
\begin{lemma}\label{lem:B4}
  We have 
 \begin{equation}
  \lim_{N\to \infty} \bbP[\cB_{(4)}^N]=1.
 \end{equation}
\end{lemma}

This time we use the following control for the martingale bracket, which is almost immediate to prove

\begin{lemma}\label{Lemma:LowBd5}
On the event $\cA_1\cap\cC_1(t)$ we have
$$\partial_t \langle A_{\cdot} \rangle_t \ge \frac {1 \wedge (\sqrt{N} A_t)}{N (\log N)^2} \;.$$
\end{lemma}
\begin{proof}

Combining Proposition \ref{Prop:BoundBracket} and Lemma \ref{Lemma:Baibi}, we obtain
$$ \partial_t \langle A_\cdot\rangle_t \ge c a_1(a_1 \wedge A_t) \ge \frac c  {\sqrt{N \log N}} \Big( \frac1{\sqrt{N \log N}} \wedge A_t \Big)\;.$$
\end{proof}

\begin{proof}[Proof of Lemma \ref{lem:B4}]
%
We only need to show that 
$$\cB_{(4)}^N\supset \cA_1\cap \cA_2 \cap \cA_4 \cap \cB_{(1)}^N \cap \cB_{(2)}^N\cap \cB_{(3)}^N.$$
A contradiction can be obtained exactly as in the proof of Lemma \ref{lem:B4}. More precisely assuming that $\cB_{(4)}^N$ does not hold 
we obtain that for some $j\in \lint K+1,L\rint$
\begin{equation}
 v_j \ge 2^{-j-1} N(\log N)^{-2} (1 \wedge N^{2-j\eta})
\end{equation}
\end{proof}

\subsection{Proof of Proposition \ref{Prop:Coalesce}} \label{sec:conclude}

We  finally conclude the proof by proving
 \begin{equation}\label{final}
\lim_{N\to \infty}\bbP[ \tau \le t_{\delta/2}+N^2+ 2\log N]=1\;.
 \end{equation}

We first need to restrict ourself to an adequate event of large probability.
Using the fact that  $(A_{\cT_L+t})_{t\ge 0}$ is a supermartingale, combinining the Martingale stopping theorem and Doob's inequality we obtain that  
\begin{equation}
 \bbP\left[ \exists t\ge 0, \ A_{\cT_L+t} \ge N^{-1} \right]\le 3 N^{5/2-L\eta}.
\end{equation}
Since Lemma \ref{Lemma:Contract}, \ref{lem:B2}, \ref{lem:B3} and \ref{lem:B4} assert that $\cT_L\le t_{\delta/2}+N^2$ with large probability, we have $\lim_{N\to \infty} \bbP[\cA^{(N)}_{5,1}]=1$ where
$$ \cA^{(N)}_{5,1} := \{ A_{t_{\delta/2}+N^2} < N^{-1} \}\;.$$
We also need to make sure that  the gradients in our dynamics are not too small.
\begin{lemma}\label{Lemma:BruteForce}
Setting 
$$\cA^{(N)}_{5,2}:= \left\{ \min_{k\in \lint 1,N-1\rint, s\in [0, 2\log N]}  \nabla X^\wedge_k(t_{\delta/2}+N^2+s) > \frac{1}{\sqrt{N} \log N} \right\},$$
we have
\begin{equation}
 \lim_{N\to \infty }\bbP\left[\cA^{(N)}_{5,2}\right]=1.
\end{equation}
\end{lemma}

\begin{proof}
First of all, by Proposition \ref{Prop:AbsCont} it suffices to control the probability of this event for the stationary process in the time interval $[0,2\log N]$, which we denote by $\bX$ for the rest of this proof. 
Recalling \eqref{tail} 
we have
\begin{equation}\label{recall}
\pi_{N,\alpha}\Big(\nabla x_k < \frac1{\sqrt N \log N}\Big) \le C N^{-1}(\log N)^{-2}\;.
\end{equation}
Now for each $k\in \lint 1,N-1 \rint$, consider the ordered set of update times $(t^{(k)}_i)_{i\ge 1}$ of the coordinate $k$ for our dynamics. We let $n_k$ be the number of updates of the site $k$ occurring in the time interval $[0,2\log N]$.

\medskip

As the dynamics is of heat-bath type, for every $i$ and $k$,  $\bX(t^{(k)}_i)$ is distributed like $\pi_{N,\alpha}$. This is also valid conditionally on the realization $(t^{(j)}_i)_{i,j}$ of the update times.
Hence considering that $\nabla X_k$ can only be altered at times $t^{(k\pm1)}_i$, from the union bound and \eqref{recall},
\begin{multline*}
 \bbP[ \exists t \in [0,2\log N], \exists k\in\lint 1, N-1\rint: \; \nabla X^\wedge_k(t) \le 1/ (\sqrt{N} \log N) \ | \ (t^{(j)}_i)_{i,j}]\\
 \le   2 C   N^{-1}(\log N)^{-2} \sum_{k=1}^{N-1} (n_{k}+1). 
\end{multline*}
Taking the expectation of the above we obtain the desired result.
\end{proof}

For every site $k\in \lint 1,N-1\rint$, let $(t^{(k)}_i)_{i=1}^{n_k}$ denote the random set of update times occurring in the interval $[t_{\delta/2}+N^2,t_{\delta/2}+N^2+2\log N]$. We set
$$ \cA^{(N)}_{5,3}:= \bigcap_{k=1}^{N-1} \{ n_k \in [1, 10\log N]\} \;.$$
The probability of $\cA^{(N)}_{5,3}$ goes to one (by a standard coupon collector argument for the lower bound while the upper bound is a direct union bound) and one can thus safely restrict oneself to the event 
$\cA_5 :=\cap_{i=1}^3 \cA_{5,i}$.\\

To conclude the proof, we are going to show that
\begin{equation}\label{showdown}
\lim_{N\to \infty} \bbP\left[ \exists k\in\lint 1,N-1  \rint, \exists i \le n_k:\;  X^{\wedge}_k(t^{(k)}_i)\ne X^{x}_k(t^{(k)}_i) \right]=0.  
\end{equation}
We are in fact going to prove an upper bound for this probability 
conditioned on both $(t^{(k)}_i)_{i=1}^{n_k}$ 
and the state of the system at the initial time $t_{\delta/2}+N^2$. We use the short-hand notation $\tilde \bbP,\,\tilde \bbE$ for this conditional probability and the corresponding expectation.\\

Say that an update at time $t^{(k)}_i$ is {\em successful} if $X^{\wedge}_k(t^{(k)}_i)= X^{x}_k(t^{(k)}_i)$. The strategy is to work recursively on the successive update times, and to use the following fact. If all the previous updates have been successful, then using Lemma \ref{Lemma:Bndq} the probability that the next update (occurring at site $k$ say) is not successful is bounded from above by (a constant times) $\delta \bar X_k$ divided by the largest gradient, this ratio being bounded by $\delta \bar X_k \sqrt{N} \log N$ on the event $\cA_{5,2}$. Since there are at most $10 \log N$ updates per site, since $\sum_k \delta \bar X_k$ is bounded by twice the area and since the area is non-increasing as long as the updates are all successful, we deduce that the probability \eqref{showdown} that  there exists an unsuccessful update is bounded by $A_{t_{\delta/2}+N^2} \sqrt{N} \log N < N^{-1/2}\log N$ on the event $\cA_{5,1}$. Thus, one can conclude. To put these heuristic observations on a firm ground, we need to introduce some notations.

Let $\tilde\tau$ be, among the set of update times $(t^{(k)}_i)_{i=1}^{n_k}$, the time of the first unsuccessful update: on the event that there are no unsuccessful updates, we set arbitrarily $\tilde\tau := t_{\delta/2} + N^2 + 2\log N$. Note that the event $\cA_{5,1}\cap\cA_{5,3}$ is measurable w.r.t.~the sigma-field generated by $(t^{(k)}_i)_{i=1}^{n_k}$ and the state of the system at the initial time $t_{\delta/2}+N^2$. Since the probability of this event goes to $1$, it suffices to show that
\begin{equation}\label{Eq:tildetau}
\lim_{N\to \infty }\P\Big( \ind_{\{\cA_{5,1}\cap\cA_{5,3}\}}\, \tilde{\P}\left(\tilde\tau < t_{\delta/2}+N^2+2\log N\right) \Big) = 0\;,
\end{equation}
in order to deduce that the merging time $\tau$ satisfies  \[\lim_{N\to \infty }\P(\tau \ge t_{\delta/2}+N^2+2\log N) = 0.\]

To prove \eqref{Eq:tildetau}, we set
$$\cA^{(i,k)}_{5,2}:= \left\{ \forall s\in [t_{\delta/2}+N^2, t^{(k)}_i),  \nabla X^\wedge_k(s) >  \frac{1}{\sqrt{N} \log N} \right\}\;.$$
Then,  almost surely
\begin{align*}
\tilde{\P}(\tilde\tau < t_{\delta/2}+N^2+2\log N) &= \tilde{\P}\big[\cup_{i,k} \{\tilde\tau = t_{i}^{(k)}\} \big]\\
&\le \tilde{\P}\big[\cup_{i,k} (\cA^{(i,k)}_{5,2})^\cc \big] + \tilde{\P}\big[\cup_{i,k} \{\tilde\tau = t_{i}^{(k)}\}\cap\cA^{(i,k)}_{5,2} \big]\\
&\le \tilde{\P}\big[(\cA_{5,2})^\cc \big] + \sum_{i,k} \tilde{\P}\big[ \{\tilde\tau = t_{i}^{(k)}\}\cap\cA^{(i,k)}_{5,2} \big]\;.
\end{align*}
The first term on the r.h.s.~goes to $0$ by Lemma \ref{Lemma:BruteForce}. Regarding the second term, we argue as follows. Recall that $\cF_t$ is the sigma-field generated by the system up to time $t$, and let $\cF_{t_-} = \sigma(\cup_{s < t} \cF_s)$. Using Lemma \ref{Lemma:Bndq} at the second line, for all $i$ and $k$ we have
\begin{align*}
\tilde{\P}\big[ \{\tilde\tau = t_{i}^{(k)}\}\cap\cA^{(i,k)}_{5,2} \big] &= \tilde{\E}\Big[\tilde{\E}\Big[\ind\left(X^{\wedge}_k(t^{(k)}_i)\ne X^{x}_k(t^{(k)}_i)\right) \,|\, \cF_{(t_{i}^{(k)})_-} \Big] \ind_{\cA^{(i,k)}_{5,2}\cap\{\tilde\tau \ge t_{i}^{(k)}\}}\Big]\\
&\le \tilde{\E}\Big[ C\frac{\delta \bar{X}_k((t^{(k)}_i)_-)}{\nabla X^\wedge_k((t^{(k)}_i)_-)} \ind_{\cA^{(i,k)}_{5,2}\cap\{\tilde\tau \ge t_{i}^{(k)}\}}\Big]\\
&\le C \sqrt N \log N\, \delta \bar{X}_k(t_{\delta/2} + N^2)\;.
\end{align*}
Consequently,
$$ \sum_{i,k} \tilde{\P}\big[ \{\tilde\tau = t_{i,k}\}\cap\cA^{(i,k)}_{5,2} \big] \le \sum_{i,k}  C\sqrt N \log N\,  \delta \bar{X}_k(t_{\delta/2} + N^2)\;,$$
and on the event $\cA_{5,1}\cap\cA_{5,3}$ this last expression is bounded by
$$C'\sqrt{N} (\log N)^2 A_{t_{\delta/2}+N^2} \le N^{-1/3}\,,$$
for some new constant $C'>0$. This concludes the proof of \eqref{Eq:tildetau}.

\section{From the top down to equilibrium}\label{Sec:AbsCont}
The goal of this section is to prove Proposition \ref{Prop:AbsCont}, that is: when started from the maximal configuration $\wedge=(N,\dots,N)$, setting $t_\delta=(1+\delta)\frac{\log N}{2\gap_N}$, one has 
\begin{align}\label{Eq:top}
\lim_{N\to\infty}\|P_{t_\delta}^\wedge-\pi_{N,\alpha}\|_{TV}= 0\,,
\end{align}
for any fixed $\delta>0$. 

Inspired by a strategy that was introduced in \cite{Lac16} in the context of the adjacent interchange process, we shall base our proof on a two-scale argument, that can be roughly described as follows. For any integer $K\geq 2$, consider the $K-1$ particles with labels $u_i:=\lfloor iN/K\rfloor$, $i=1,\dots,K-1$. These will be called the {\em special particles}. The proof of  \eqref{Eq:top} consists of  three main steps. 

\medskip
{\bf Step 1}. Starting from $\wedge$, after a time $t=t_{\delta/2}$, if $K$ is fixed and $N$ tends to $\infty$, then the joint distribution of the positions of the special particles  
\begin{align}\label{Eq:topk}
Y_i(t) = X_{u_i}(t)\,,\qquad i=1,\dots,K-1
\end{align}
is arbitrarily close to the corresponding equilibrium distribution, see Proposition \ref{prop:special} below. This step is based on a subtle use of the FKG inequality together with the control of the expected value of the variables $Y_i(t)$. 

\medskip
{\bf Step 2}.
Consider the censored dynamics obtained by freezing the positions of the special particles and letting the rest of the particles evolve as usual. We will show that if $K$ is taken proportional to $\delta^{-1}$, uniformly in the initial condition, the censored dynamics 
at time $s_\delta:=\delta\,\frac{\log N}{2\gap_N}$ has essentially reached the conditional equilibrium given by $\pi_{N,\alpha}$ conditioned on the positions of the special particles. For this step it will be sufficient to exploit an upper bound on the mixing time that is tight up to a constant factor as e.g.\ the one obtained in \cite{RW05} in the case $\ga=1$, see \eqref{Eq:roughUB} above.  

\medskip
{\bf Step 3}.
We combine the results in the previous two steps to obtain the desired conclusion. The key point is that the distance to equilibrium at time $t_\delta=t_{\delta/2}+s_{\delta/2}$ appearing in \eqref{Eq:top} satisfies 
\begin{align}\label{Eq:censor}
\|P_{t_\delta}^\wedge-\pi_{N,\alpha}\|_{TV}\leq \|P_{t_\delta,*}^{\wedge}-\pi_{N,\alpha}\|_{TV}\,,
\end{align}
where the distribution $P_{t_\delta,*}^{\wedge}$ is obtained by running the standard dynamics, starting from $\wedge$, for a time
$t_{\delta/2}$ and then by running from there the censored dynamics for a time $s_{\delta/2}$. This step requires an adaptation to our continuous setting of the so called censoring inequality of Peres and Winkler \cite{PWcensoring}.   

We start developing 
the above program with a discussion of the censoring inequality. We then move to the proof of the mentioned steps in the given order.  

\subsection{Censoring lemma}\label{sec:censor}
The censoring inequality established by Peres and Winkler \cite{PWcensoring} allows one to compare the distance to equilibrium at time $t$ for the process under consideration with the distance to equilibrium at time $t$ for a censored process in which some of the updates have been omitted according to a given censoring scheme. 
In the context of Glauber dynamics for monotone, finite state spin systems, their argument rests on the following two key properties of a monotone dynamics:  If the initial state has a distribution $\nu$ whose density w.r.t.\ the equilibrium measure $\pi$ is increasing, then for any $t\geq 0$, the distribution $\nu P_t$ of the state at time $t$ satisfies
\begin{enumerate}[1)]
\item  $\nu P_t$  has an increasing density w.r.t. $\pi$,
\item  $\nu P_t$ is stochastically lower than the distribution of the state of the censored dynamics, say $\nu P_{t,*}$, for any valid censoring scheme.
\end{enumerate}
Properties 1 and 2 then allow one to prove the censoring inequality
\begin{align}\label{Eq:censor2}
\|\nu P_{t}-\pi\|_{TV}\leq \|\nu P_{t,*}-\pi\|_{TV}.
\end{align}
We shall follow the same line of reasoning here.
However, 
a technical problem arises with respect to the usual discrete spin setting: to prove the censoring inequality \eqref{Eq:censor} we need to start with the distribution $\nu=\delta_{\wedge}$ which has no density w.r.t.\ equilibrium, while the strategy outlined above is 
crucially based on the existence of such a density. 
Notice also that $P_{t}^\wedge = \delta_\wedge P_t$ has no density w.r.t. equilibrium at any time $t\geq 0$ and so one cannot get around this problem by regularising the measure with a burn-in time. We shall need a more general version of the above properties   
which extends to a certain family of measures with a singular part. 

We start by defining the latter. For $k\in\lint1,N\rint$, define the nested sets 
\begin{equation}\label{Eq:omegakn}
\gO_{k,N} = \left\{
 x\in\gO_N: \,x_{k}= N\right\}.
\end{equation}
Thus $\gO_{1,N}=\{\wedge\}$ is the maximal configuration, while $\gO_{N,N} = \gO_N$ is the whole set of particle positions. 
Let $\pi_{k,\alpha}$ denote the probability measure supported on $\gO_{k,N}$ defined as the law of the random vector $(x_1,\dots,x_{N-1})$ of the partial sums $x_j=\sum_{i=1}^j\eta_j$ where $\eta_1,\dots,\eta_k$ are i.i.d.\ with distribution $\gG(\ga,\gl)$, for some arbitrary $\gl>0$,  conditioned to $\sum_{i=1}^k\eta_i=N$, while $\eta_{k+1},\dots,\eta_N\equiv 0$.  Notice  that this notation is consistent with our notation $\pi_{N,\alpha}$ for the equilibrium measure on $\gO_N$. 
If a probability measure $\mu$ supported on $\gO_{k,N}$ is absolutely continuous w.r.t.\ $\pi_{k,\alpha}$, we write $\frac{d\mu}{d\pi_{k,\alpha}}$ for the corresponding density, and say that $\mu$ belongs to the family $\cS_k$ if $\frac{d\mu}{d\pi_{k,\alpha}}$ is an increasing function. In particular, $\cS_1=\{\delta_\wedge\}$, while $\cS_N$ coincides with the set of distributions on $\gO_N$ with an increasing density with respect to equilibrium. 
Finally, we define the family of measures $\cS$ consisting of all probability measures $\mu$ on $\gO_N$ such that 
\begin{align}\label{Eq:familyS}
\mu=\sum_{k=1}^N \gamma_k \mu_k\,,\qquad \mu_k\in\cS_k\,,
\end{align}
for some $\gamma_k\geq 0$ with $\sum_{k=1}^N\gamma_k=1$. 
Notice that \eqref{Eq:familyS} is a decomposition into mutually singular measures, since if $1\leq j<k\leq N$, then $\mu_j(\gO_{j,N})=1$ while $\mu_k(\gO_{j,N})=0$. 
The following lemma is a key fact about the set $\cS$.
\begin{lemma}\label{lem:kfs}
If $\mu,\nu$ are two probability measures on $\gO_N$ such that $\mu\in\cS$ and $\mu\leq \nu$, then 
 \begin{align}\label{Eq:familyS2}
 \|\mu-\pi_{N,\alpha}\|_{TV}\leq \|\nu-\pi_{N,\alpha}\|_{TV}.
\end{align}
\end{lemma}
\begin{proof}
Write $$\mu=\sum_{j=1}^N\gamma_j \mu_j=(1-\gamma_N)\mu' + \gamma_N \mu_N,$$ where $\mu'=(1-\gamma_N)^{-1}\sum_{j=1}^{N-1}\gamma_j\mu_j$ is singular w.r.t.\ $\pi_{N,\alpha}$ while $\mu_N$ is absolutely continuous w.r.t.\ $\pi_{N,\alpha}$, with an increasing density $\varphi_N$. We have
\begin{align*}
\|\mu-\pi_{N,\alpha}\|_{TV} &= \sup_{C\in \cB(\gO_{N-1,N}), A\in \cB(\gO_{N-1,N}^\cc)} \mu(A\cup C) - \pi_{N,\alpha}(A\cup C)\\
&=  \sup_{C\in \cB(\gO_{N-1,N}), A\in \cB(\gO_{N-1,N}^\cc)} (1-\gamma_N) \mu'(C) + \gamma_N \mu_N(A) - \pi_{N,\alpha}(A)\\
&= 1-\gamma_N + \sup_{ A\in \cB(\gO_{N-1,N}^\cc)} \gamma_N \mu_N(A) - \pi_{N,\alpha}(A)\;.
\end{align*}
We then define $A=\{x\in\gO_N: \,\varphi_N(x)\geq 1/\gamma_N\}$. It is easy to check that this event maximises the second term on the r.h.s.~of the last equation. Therefore, setting $B=A\cup\gO_{N-1,N}$, and observing that one has $\mu'(B)=\mu'(\gO_{N-1,N})=1$ and $\mu_N(B)=\mu_N(A)$, $\pi_{N,\alpha}(B)=\pi_{N,\alpha}(A)$, we deduce that
\begin{align*}
\|\mu-\pi_{N,\alpha}\|_{TV} &= \mu(B) - \pi_{N,\alpha}(B)\;.
\end{align*}
Since $A$ and $\gO_{N-1,N}$ are increasing, the set $B$ is also increasing. Therefore, $\mu(B)\leq \nu(B)$ and 
$$
\|\mu-\pi_{N,\alpha}\|_{TV} \leq \nu(B)-\pi_{N,\alpha}(B)\leq \|\nu-\pi_{N,\alpha}\|_{TV}.
$$
\end{proof} 
Let $Q_i:L^2(\gO_N,\pi_{N,\alpha})\mapsto L^2(\gO_N,\pi_{N,\alpha})$, $i=1,\dots,N-1$, denote the orthogonal projection onto functions that do not depend on the position of the $i$-th particle:
$$
Q_i f(x) = \pi_{N,\alpha}[f\,|\,x_j, j\neq i]\,.
$$
If $\mu$ is a probability on $\gO_N$, 
we write $\mu Q_i$ for the probability measure defined by
$$
\mu Q_i (f) = \int \mu(dx)Q_i f(x)\,.
$$
This is the distribution obtained from $\mu$ after one update at $i$.
\begin{lemma}\label{lem:muqi}
If $\mu  \in\cS$ then for any $i=1,\dots,N-1$
\begin{enumerate}[1)]
\item  $\mu Q_i\in\cS$;
\item  $\mu Q_i$ is stochastically lower than $\mu$.
\end{enumerate}
\end{lemma}
\begin{proof}
Pick $\mu_k\in\cS_k$. If $i>k$, then $\mu_k Q_i=\mu_k$, since $x_k=N$ forces $x_j=N$ for all $j>k$.
If $i<k$, writing $\varphi_k = \frac{d\mu_k}{d\pi_{k,\alpha}}$, 
for any bounded measurable $f$ one has 
$$
\mu_k Q_i (f) = \pi_{k,\alpha}\left[\varphi_kQ_if\right]=\pi_{k,\alpha}\left[(Q_i\varphi_k) f\right],
$$
where we note that if $x_k=N$ and  $i<k$ then $Q_if(x)=\pi_{k,\alpha}[f\,|\,x_j, j\neq i]$, and therefore $Q_i$ is self-adjoint in $L^2(\gO_{k,N},\pi_{k,\alpha})$. Thus, $\mu_k Q_i$ has density $Q_i\varphi_k$ w.r.t.\ $\pi_{k,\alpha}$. Recall the notation $x^{(i,u)}$ for the configuration $x$ updated at the $i$-th particle. For any $x,y\in\gO_{k,N}$ with $x\leq y$, using $x^{(i,u)}\leq y^{(i,u)}$ for all $u\in[0,1]$, it follows that 
\begin{align*}
Q_i\varphi_k(x) &= \int_0^1 \varphi_k\left(x^{(i,u)}\right)\rho_\ga(u)du\\&\leq \int_0^1 \varphi_k\left(y^{(i,u)}\right)\rho_\ga(u)du=Q_i\varphi_k(y)\,.
\end{align*}
In other words $Q_i\varphi_k$ is increasing, and $\mu_k Q_i\in \cS_k$ if $i<k$.
When $i=k$,  observe that if $x_{k+1}=N$, then $ Q_k (f) = \pi_{k+1,\alpha}\left[f\,|\,x_j,j<k\right]$. Therefore, if $\psi_{k,\alpha}$ denotes the density of the marginal of $\pi_{k,\alpha}$ on $(x_1,\dots,x_{k-1})$ w.r.t.\ the marginal of $\pi_{k+1,\alpha}$ on the same variables, 
\begin{align*}
\mu_k Q_k (f) &= \pi_{k,\alpha}\left[\varphi_k\pi_{k+1,\alpha}\left[f\,|\,x_j,j<k\right]\right] \\&
= \pi_{k+1,\alpha}\left[\psi_{k,\alpha}\varphi_k\pi_{k+1,\alpha}\left[f\,|\,x_j,j<k\right]\right]\\&
= \pi_{k+1,\alpha}\left[\pi_{k+1,\alpha}\left[\psi_{k,\alpha}\varphi_k\,|\,x_j,j<k\right] f\right]\\&
= \pi_{k+1,\alpha}\left[\psi_{k,\alpha}\varphi_kf\right].
\end{align*}
This shows that $\mu_k Q_k$ is supported on $\gO_{k+1,N}$ and has density $\psi_{k,\alpha}\varphi_k$ w.r.t.\ $  \pi_{k+1,\alpha}$. A direct computation shows that 
$$
\psi_{k,\alpha} = C_{\ga,k,N}(N-x_{k-1})^{-\ga}\,,
$$
for some positive constant $C_{\ga,k,N}$. In particular, $\psi_{k,\alpha}$ is increasing for any $\ga>0$. It follows that if $\mu_k\in\cS_k$, then $\mu_kQ_k\in\cS_{k+1}$, for all $k=1,\dots,N-1$.
Taking a generic $\mu\in\cS$, by linearity the above implies that $\mu Q_i\in\cS$ for any $ i=1,\dots,N-1$.

To prove the stochastic domination $\mu Q_i\leq \mu$, for $\mu\in\cS$, it is sufficient to show that $\mu_k Q_i\leq \mu_k$, for $\mu_k\in\cS_k$, for all $k,i$. 
Pick $\mu_k\in\cS_k$ for some $k=1,\dots,N$ and an increasing function $g$ on $\gO_N$. We are going to show that $\mu_k Q_i (g)\leq \mu_k (g)$ for any $i=1,\dots,N-1$. If $i>k$ then $\mu_kQ_i=\mu_k$ and there is nothing to prove. If $i<k$, as above we may write
$$
\mu_k Q_i (g) = \pi_{k,\alpha}\left[(Q_i\varphi_k) g\right] = \pi_{k,\alpha}\left[(Q_i\varphi_k) (Q_ig)\right].
$$ 
Since $\varphi_k$ is also increasing, the FKG inequality on $\bbR$, which is valid for any probability measure, implies that $(Q_i\varphi_k) (Q_ig)\leq Q_i(\varphi_k g)$ pointwise. Therefore
$$
\mu_k Q_i (g) \leq \pi_{k,\alpha}\left[Q_i(\varphi_k g)\right] = \pi_{k,\alpha}\left[\varphi_k g\right] = \mu_k (g).
$$ 
Finally, if $i=k$, then as before we have
$$
\mu_k Q_k (g) = \pi_{k+1,\alpha}\left[\psi_{k,\alpha}\varphi_kg\right].
$$
On the other hand, defining the function $\bar g(x):=g(x_1,\dots,x_{k-1},N,\dots,N)$, one has
$$
\mu_k (g) = \pi_{k+1,\alpha}\left[\psi_{k,\alpha}\varphi_k \bar g\right].
$$
Thus, the conclusion $\mu_k Q_k (g) \leq \mu_k (g)$ follows from the fact that $g\leq \bar g$. 
\end{proof}
In the next lemma we consider the effect of a sequence of updates on a measure $\mu\in\cS$ and compare it with the effect of another sequence obtained from the first by removing some of the updates.  

\begin{lemma}\label{lem:muscheme}
Pick $n\in\bbN$ and fix a sequence $z:=(z_1,\dots,z_n)\in\lint1,N-1\rint^n$.
For any $\mu\in\cS$, if $\mu^z$ denotes the new measure 
\begin{align}\label{Eq:muz}
\mu^z=\mu\, Q_{z_1}\cdots Q_{z_n}\,,
\end{align}
then $\mu^z\in\cS$. Moreover, if $z'$ denotes a sequence obtained from $z$ by removing some of  the entries, then $\mu^z\leq \mu^{z'}$ and
\begin{align}\label{Eq:sched}
\|\mu^z-\pi_{N,\alpha}\|_{TV}\leq \|\mu^{z'}-\pi_{N,\alpha}\|_{TV}
\end{align}
\end{lemma}
\begin{proof}
Lemma \ref{lem:muqi} shows that $\mu^z\in \cS$ for any $\mu\in\cS$ and any sequence $z$. For the second part of the lemma, by a telescoping argument it is sufficient to consider the case where $z$ and $z'$ differ by the removal of a single update, say $z_j$, so that 
$$
z=(z_1,\dots,z_{j-1},z_j,z_{j+1},\dots,z_n)\,,\quad z'=(z_1,\dots,z_{j-1},z_{j+1},\dots,z_n)
$$
Let $\mu_1 = \mu Q_{z_1}\cdots Q_{z_{j}}$, and $\mu_2 = \mu Q_{z_1}\cdots Q_{z_{j-1}}$. Then $\mu_1=\mu_2Q_{z_{j}}$ and thus, by Lemma \ref{lem:muqi} one has  $\mu_1\leq \mu_2$.
Moreover,  
$$
\mu^z = \mu_1Q_{z_{j+1}}\cdots Q_{z_{n}} \leq \mu_2Q_{z_{j+1}}\cdots Q_{z_{n}} = \mu^{z'}\,, 
$$ 
where the inequality follows from the fact that each update preserves the monotonicity, see Proposition \ref{grandgradient}. 
The conclusion follows from Lemma \ref{lem:kfs}.
\end{proof}

We can now state and prove the censoring inequality in our setup. A {\em censoring scheme} $\cC$ is defined as a c\`adl\`ag map
$$
\cC:[0,\infty)\mapsto \cP(\{1,\dots,N-1\}),
$$
where $\cP(A)$ denotes the set of all subset of a set $A$. The subset $\cC(s)$, at any time $s\geq 0$, represents the set of particles whose update is to be suppressed at that time. More precisely, given a censoring scheme $\cC$, and an initial condition $x\in\gO_N$, we write $P_{t,\cC}^x$ for the law of the random variable obtained by starting at $x$ and applying the standard graphical construction (see Proposition \ref{grandgradient}) with the proviso
that if the particle with label $j$ rings at time $s$, then the update is performed if and only if $j\notin\cC(s)$.  In particular, the uncensored evolution $P^x_t$ corresponds to $P^x_{t,\cC}$ when $\cC(s)\equiv\eset$. Given a distribution $\mu$ on $\gO_N$, we write $$\mu P_{t,\cC}=\int  P^x_{t,\cC}\,\mu(dx).$$ 
  
\begin{lemma}\label{cor:censor}
For any $\ga>0$, if $\mu\in\cS$, and $\cC$ is a censoring scheme, then for all $t\geq 0$:
\begin{enumerate}[1)]
\item $\mu P_t\in \cS$ and $\mu P_{t,\cC}\in \cS$,
\item  $\mu P_{t,\cC}$ is stochastically higher than $\mu P_t$.
\end{enumerate}
Moreover, for all $t\geq 0$:  
\begin{equation}\label{Eq:censura}
\|\mu P_t-\pi_{N,\alpha}\|_{TV}\leq \|\mu P_{t,\cC}-\pi_{N,\alpha}\|_{TV}.
\end{equation}
\end{lemma}

\begin{proof}
It is sufficient to prove 1) and 2) above, since the conclusion \eqref{Eq:censura} is then a consequence of Lemma \ref{lem:kfs}. To prove 1) and 2) note that by conditioning on the realization $\cT_t$ of the Poisson clocks $\cT^{(j)}$, $j\in\lint 1,N-1\rint $ up to time $t$ in the graphical construction, one has that  
the uncensored and the censored evolution are measures of the form $\mu^z$ and $\mu^{z'}$ respectively; see \eqref{Eq:muz}. By Lemma \ref{lem:muscheme} one has $\mu^z,\mu^{z'}\in\cS$ and $\mu^{z}\leq \mu^{z'}$. Taking the expectation over $\cT_t$ shows that $\mu P_t\in \cS$ and $\mu P_{t,\cC}\in \cS$, and that $\mu P_t\leq \mu P_{t,\cC}$. 

 \end{proof}

\subsection{Relaxation of the special particles}\label{sec:special}
Here we show that special particles have reached equilibrium by time $t_{\delta/2}$; see Proposition \ref{prop:special}.
The key to this result will be Proposition \ref{prop:muW} below.
Recall the notation $\cS_N$ introduced in Section \ref{sec:censor} for the set of probability measures on $\gO_N$ with an  increasing density w.r.t.\ $\pi_{N,\alpha}$. Given a probability $\mu$ on $\gO_N$, we write $\bar \mu$ for the marginal of $\mu$ on the special particle positions $y:=(y_1,\dots,y_{K-1})$. 

\begin{lemma}
\label{lem:incmarg}
If $\mu\in\cS_N$ then $\bar \mu$ is absolutely continuous w.r.t.\ $\bar \pi_{N,\alpha}$ and the corresponding density is increasing on $\bbR^{K-1}$.
\end{lemma}
\begin{proof}
Let $\varphi = d\mu/d\pi_{N,\alpha}$. The density of $\bar \mu$ w.r.t.\ $\bar\pi_{N,\alpha}$ is given by the conditional expectation
$$
\bar\varphi(y) = \pi_{N,\alpha}(\varphi |y).
$$
To prove that it is increasing, we have to show that $\pi_{N,\alpha}(\varphi |y)\geq \pi_{N,\alpha}(\varphi |y')$ whenever $y\geq y'$. The latter domination can be seen as follows. Let $x$ be the highest configuration of particle positions such that $y_i=x_{u_i}$, $i=1,\dots,K-1$, and let $x'$   be the lowest configuration of particle positions such that $y'_i=x'_{u_i}$, $i=1,\dots,K-1$. Clearly, $x\geq x'$. Then use $(x,x')$ as initial conditions in the graphical construction (Proposition \ref{grandgradient}) for the censored dynamics where all updates of the special particles are suppressed.
As time goes to infinity the two distributions  converge weakly to  $\pi_{N,\alpha}(\cdot |y), \pi_{N,\alpha}(\cdot |y')$ respectively. Since $x\geq x'$ and the graphical construction preserves the order, this shows that $  \pi_{N,\alpha}(\cdot |y)\geq \pi_{N,\alpha}(\cdot |y')$.
\end{proof}
We  use the following notation for the centered height of the special particles:
$$
w_i = x_{u_i} - u_i\,,\quad W= \sum_{i=1}^{K-1}w_i,
$$
and write $\mu(W)= \bar \mu(W)$ for the expected value of $W$. Note that at equilibrium, for $N$ large, the vector $(w_1,\dots,w_{K-1})$ behaves roughly as a normal vector, and the fluctuations of $W$ are of order $\sqrt N$ for each fixed $K$.   
The results below are valid for all $\ga\geq 1$. 

\begin{proposition}\label{prop:muW}
For any $\gep>0$, $K\in\bbN$, there exists $\eta=\eta(K,\gep)>0$ such that for all $N\geq 2$, $\mu\in\cS_N$ one has:
\begin{equation}\label{Eq:contr1}
\mu(W)\leq\eta \sqrt N\;\;\;\Rightarrow\;\;\;\|\bar \mu -\bar \pi_{N,\alpha}\|_{TV}\leq \gep.
\end{equation}
\end{proposition}
\begin{proof}
We follow \cite[Section 5]{Lac16}, where a similar statement was proved in the context of random permutations. Given a constant $\gl>0$, define the events
\begin{align*}
A_i = \left\{x:w_i\geq \gl\sqrt N\right\}\,,\qquad A = \bigcap_{i=1}^{K-1}A_i\,,\qquad B = \bigcup_{i=1}^{K-1}A_i\,. 
\end{align*}
Let us first show that for any $\mu\in\cS_N$, $a>0$
\begin{equation}\label{Eq:contr2}
\mu(A)\geq (1+a) \pi_{N,\alpha}(A)\;\;\;\Rightarrow\;\;\;\mu(W)\geq a\gl\sqrt N\pi_{N,\alpha}(A).
\end{equation}
Let $\varphi$ denote the density $d\mu/d\pi_{N,\alpha}$. The sets $A_i$ and $A_i^\cc$ are stable under the operations $\wedge$ and $\vee$ introduced in \eqref{minimax}, and the  FKG inequality applied to $\pi_{N,\alpha}(\cdot|A_i)$ shows that 
\begin{align*}
\mu(w_i; A_i) &:= \mu(w_i \ind_{A_i})\\
&= \pi_{N,\alpha}(A_i)\pi_{N,\alpha}(\varphi w_i| A_i) \\&\geq \pi_{N,\alpha}(\varphi ; A_i) \pi_{N,\alpha}(w_i| A_i) = \mu(A_i) \pi_{N,\alpha}(w_i| A_i).
\end{align*}
Similarly, the FKG inequality for $\pi_{N,\alpha}(\cdot|A_i^\cc)$ shows that $$\mu(w_i; A^\cc_i)\geq \mu(A^\cc_i) \pi_{N,\alpha}(w_i| A^\cc_i).$$
Therefore, using $\pi_{N,\alpha}(w_i)=0$:
\begin{gather*}
\mu(w_i) \geq (\mu(A_i)- \pi_{N,\alpha}(A_i)) \pi_{N,\alpha}(w_i| A_i)+(\mu(A^\cc_i)- \pi_{N,\alpha}(A^\cc_i)) \pi_{N,\alpha}(w_i| A^\cc_i).
\end{gather*}
Since $A_i$ and $w_i$ are increasing and since $\pi_{N,\alpha}(w_i) = 0$, one has $(\mu(A^\cc_i)- \pi_{N,\alpha}(A^\cc_i)) \pi_{N,\alpha}(w_i| A^\cc_i)\geq 0$ and therefore
\begin{equation}\label{Eq:contr3}
\mu(w_i) \geq (\mu(A_i)- \pi_{N,\alpha}(A_i)) \pi_{N,\alpha}(w_i| A_i)\geq (\mu(A_i)- \pi_{N,\alpha}(A_i))\gl\sqrt N.
\end{equation}
Consider the function $$\psi = \left(\sum_{i=1}^{K-1} {\bf 1}_{A_i}\right) - {\bf 1}_{A}.$$
Then $\psi$ is increasing, and applying FKG we obtain
\begin{align*}
\sum_{i=1}^{K-1}(\mu(A_i)- \pi_{N,\alpha}(A_i)) &= \mu(A)- \pi_{N,\alpha}(A) + \pi_{N,\alpha}\left((\varphi-1)\psi \right) 
\\&\geq \mu(A)- \pi_{N,\alpha}(A).
\end{align*}
Summing in \eqref{Eq:contr3}, 
\begin{gather*}
\mu(W) \geq  (\mu(A)- \pi_{N,\alpha}(A))\gl\sqrt N.
\end{gather*}
This proves \eqref{Eq:contr2}.
Next, let us show that for any $\mu\in\cS_N$, $a>0$
\begin{equation}\label{Eq:contr4}
\mu(A)\leq (1+a) \pi_{N,\alpha}(A)\;\;\;\Rightarrow\;\;\;\|\bar\mu-\bar \pi_{N,\alpha}\|_{TV}\leq a+\pi_{N,\alpha}(B).
\end{equation}
The sets $A,B$ satisfy the assumptions of  Lemma \ref{lem:piapib}. Therefore, 
$$
\frac{\mu(A)}{\pi_{N,\alpha}(A)}=\pi_{N,\alpha}(\varphi | A) \geq \pi_{N,\alpha}(\varphi | B)=\frac{\mu(B)}{\pi_{N,\alpha}(B)} .
$$
If $\mu(A)\leq (1+a) \pi_{N,\alpha}(A)$, then 
\begin{equation}\label{Eq:contr5}
\mu(B)\leq (1+a) \pi_{N,\alpha}(B).
\end{equation}
From Lemma \ref{lem:incmarg} we know that $\bar\varphi=d\bar\mu/d\bar \pi_{N,\alpha}$ is increasing. If $x,x'$ are two particle configurations such that $x'\in B^\cc$ and $x\in A$, then $y_i=x_{u_i}\geq x'_{u_i}=y'_i$, for all $i=1,\dots,K-1$.
Since $A,B$ are measurable w.r.t.\ the $y$ variables, we write $\bar A,\bar B$ for the corresponding subsets of $\bbR^{K-1}$, so that $x\in A \iff y\in \bar A$ and $x\in B\iff y\in \bar B$. Thus, we have
$$\bar\varphi(y')\leq \bar\varphi(y)\,,\qquad  y'\in \bar B^\cc, \;y\in \bar A.
$$ 
Integrating over $y$  w.r.t.\ $\bar \pi_{N,\alpha}(\cdot | \bar A)$  in the above inequality one finds 
$$
\bar\varphi(y')\leq \frac{\bar\mu(\bar A)}{\bar \pi_{N,\alpha}(\bar A)}=\frac{\mu(A)}{\pi_{N,\alpha}(A)}\leq 1+a\,,\qquad  y'\in \bar B^\cc,
$$
where we have used the assumption $\mu(A)\leq (1+a)\pi_{N,\alpha}(A)$. In conclusion, 
\begin{align*}
\|\bar\mu-\bar \pi_{N,\alpha}\|_{TV} &= \int\left(\bar\varphi(y)-1\right)_+\bar \pi_{N,\alpha}(dy)\\&\leq  \int_{\bar B}\bar\varphi(y)\bar \pi_{N,\alpha}(dy)+ \int_{\bar B^\cc}\left(\bar\varphi(y')-1\right)_+\bar \pi_{N,\alpha}(dy')\\&
\leq \bar\mu(\bar B) + a \,\bar \pi_{N,\alpha}(\bar B^\cc) = \mu(B) + a \, \pi_{N,\alpha}(B^\cc) .
\end{align*}
Using \eqref{Eq:contr5} we obtain \eqref{Eq:contr4}.

Finally, we need a lower bound on the probability  $\pi_{N,\alpha}(A)$ and an upper bound on the probability $\pi_{N,\alpha}(B)$. At equilibrium $w_i$ is approximately normal with mean zero and variance $iN/K$, and therefore $$b_1(\gl,K)\leq \pi_{N,\alpha}(A_i)\leq b_2(\gl,K)\, $$ 
for all $i=1,\dots,K-1$, $N\geq 2$, for some constants $b_2(\gl,K)\to 0$ as $\gl\to\infty$ and $b_1(\gl,K)>0$ for all $\gl>0$. 
As for the lower bound on $\pi_{N,\alpha}(A)$ observe that by the FKG inequality one has
$$
\pi_{N,\alpha}(A)\geq \prod_{i=1}^{K-1}\pi_{N,\alpha}(A_i)\geq b_1(\gl,K)^{K-1}.
$$
On the other hand, 
$$
\pi_{N,\alpha}(B)\leq \sum_{i=1}^{K-1}\pi_{N,\alpha}(A_i)\leq (K-1)b_2(\gl,K).
$$
Then, for any $\gep>0$, any fixed $K$, taking $\gl$ large enough, we find that there exists a constant $\gd_1=\gd_1(\gep,K)>0$ such that 
\begin{equation}\label{Eq:contr6}
\gd_1\leq \pi_{N,\alpha}(A)\leq \pi_{N,\alpha}(B)\leq \frac\gep2.
\end{equation}
Once we have  \eqref{Eq:contr2}, \eqref{Eq:contr4} and \eqref{Eq:contr6} we can conclude as follows.
Suppose that $\mu(W)\leq \eta\sqrt N$. If $a = 2\eta/\gl\gd_1$, then by \eqref{Eq:contr2} and \eqref{Eq:contr6} one must have $\mu(A)\leq (1+a) \pi_{N,\alpha}(A)$. 
Therefore by \eqref{Eq:contr4} and \eqref{Eq:contr6}
$$
\|\bar\mu-\bar \pi_{N,\alpha}\|_{TV}\leq \frac{2\eta}{\gl\gd_1}+\frac\gep2.
$$
The desired conclusion follows by taking $\eta=\gep\gl\gd_1/4$.
\end{proof}
Next, we address the problem of controlling the expected value of $W$ at time $t$ when started from the maximal configuration. 
 \begin{proposition}\label{prop:Wt}
For any $k=1,\dots,N-1$, any $t\geq 0$:
$$
\bbE\left[X^\wedge_k(t)\right]\leq k + 2Ne^{-\gap_Nt}.
$$
In particular, if $\mu_t=\gd_\wedge P_t$, then for all $t\geq 0$: 
\begin{equation}\label{Eq:contrW1}
\mu_t(W)\leq 2KN e^{-\gap_Nt}.
\end{equation}
\end{proposition}
\begin{proof}
Defining 
$$
\varphi_j(k) = \sqrt{\frac2N}\sin\left(\frac{ jk \pi}{N}\right)\,,
$$
one has for all  $j,k=1,\dots,N-1$:
\begin{equation}\label{Eq:diagD}
\Delta \varphi_j (k) = \varphi_j (k+1) + \varphi_j (k-1)-2\varphi_j (k) = -2\gl_j \varphi_j (k)\,,
\end{equation}
where 
\begin{equation}\label{Eq:diagD2}
\gl_j = 1-\cos\left(\frac{ j \pi}{N}\right)\,.
\end{equation}
Set $v(t)=(v_1(t),\dots,v_{N-1}(t))$, where $$v_k(t)=\bbE\left[X^\wedge_k(t)\right] - k.$$
Expanding the vector $v(t)$ in the orthonormal basis $\varphi_j$, $j=1,\dots,N-1$, one finds 
$v_k(t) = \sum_{j=1}^{N-1} a_j(t)  \varphi_j(k)$, where 
$a_j(t)
=\sum_{k=1}^{N-1}\varphi_j(k)v_k(t)$. Using \eqref{genlinear} one finds
$$
a_j(t) = a_j(0)e^{-\gl_j t}\,,\qquad a_j(0) =\sum_{k=1}^{N-1}
\varphi_j(k)(N-k).
$$
In particular, $|a_j(0)|\leq N^{3/2}/ \sqrt 2$, and
$$
v_k(t)\leq N\sum_{j=1}^{N-1}e^{-\gl_j t}\,.
$$
Using $\gl_j \geq j\gl_1$ it follows that
$$
v_k(t)\leq \frac{Ne^{-\gl_1 t}}{1-e^{-\gl_1 t}}.
$$ 
If $t$ is such that $e^{-\gl_1 t}\leq 1/2$ then this implies $v_k(t)\leq 2Ne^{-\gl_1 t}$. On the other hand if 
$e^{-\gl_1 t}\geq 1/2$ then clearly $v_k(t)\leq N\leq 2Ne^{-\gl_1 t}$. Since $\gl_1=\gap_N$, this proves the desired upper bound. 
\end{proof}

Next we want to use the bound in Proposition \ref{prop:Wt}  to obtain, via Proposition \ref{prop:muW}, the desired control on the convergence of special particles. However, again a technical problem arises due to the fact that $\mu_t= \gd_\wedge P_t$ has no density w.r.t.\ equilibrium. We overcome this by showing that the singular part of $ \mu_t$ has very small mass if $t$ is large. 

\begin{lemma}\label{lem:mutn}
For any $t\geq 0$, the measure $\mu_t= \gd_\wedge P_t$ satisfies 
 \begin{equation}\label{Eq:decomp}
\mu_t = (1-\gamma_{t,N})\mu'_t + \gamma_{t,N}\mu_{t,N}\,,
\end{equation}
where $\mu_{t,N}\in\cS_N$,  and $\gamma_{t,N}\geq 1- Ne^{-{t/N}}$, for some probability measure $\mu'_t$. 
\end{lemma}
\begin{proof}
We know from Lemma \ref{cor:censor} that $\mu_t\in\cS$ for all $t\geq 0$. Thus
$$
\mu_t = (1-\gamma_{t,N})\mu'_t + \gamma_{t,N}\mu_{t,N},
$$
with $\mu_{t,N}\in\cS_N$ for some coefficients $ \gamma_{t,N}$, and some singular measure $\mu'_t$.
It remains to show that $\gamma_{t,N}\geq 1- Ne^{-{t/N}}$.
By conditioning on the realization $\cT_t$ of the Poisson clocks $\cT^{(j)}$, $j\in\lint 1,N-1\rint $ up to time $t$ in the graphical construction, one has that the distribution of the particles at time $t$ is of  the form $\mu^z$, for some sequence of updates $z$; see \eqref{Eq:muz}. In the proof of Lemma \ref{lem:muqi} we have seen that if $\mu\in \cS_k$, then $\mu Q_k\in \cS_{k+1}$, and therefore if the sequence $z$ contains the full sweep $(1,2,\dots,N-1)$ as a subsequence, then $\mu^z\in\cS_N$. Let $E_t$ denote the event that $z$ contains $(1,2,\dots,N-1)$ as a subsequence. Since $\cS_N$ is stable under convex combinations, 
taking the expectation over $\cT_t$ one finds that $ \gamma_{t,N} \geq\bbP(E_t)$. A rough lower bound  on the latter can be obtained by dividing the interval $[0,t]$ in $N-1$ intervals and by requiring that for each $i$ the clock of particle with label $i$ rings during the $i$-th time interval.    This shows that $\bbP(E_t)\geq 1- Ne^{-{t/N}}$.
\end{proof}

We are ready to accomplish the first and most delicate step in the program outlined at the beginning of this section.   
\begin{proposition}\label{prop:special}
Fix $\ga\geq 1$ and $K\in\bbN$. Let $\mu_t= \gd_\wedge P_t$ and let $\bar \mu_t$ denote the marginal of $\mu_t$ on the special particle positions $Y_i(t), i=1,\dots,K-1$. If $\bar\pi_{N,\alpha}$ denotes the corresponding equilibrium distribution, for any fixed $\gd>0$, with $t_\gd = (1+\gd)\frac{\log N}{2\gap_N}$ one has
\begin{equation}\label{Eq:top1}
\lim_{N\to\infty}\|\bar \mu_{t_\gd} - \bar\pi_{N,\alpha}\|_{TV} = 0. 
\end{equation}
\end{proposition}
\begin{proof}
From Lemma \ref{lem:mutn} it follows that
$$
\|\bar \mu_{t_\gd} - \bar\pi_{N,\alpha}\|_{TV}\leq Ne^{-t_\gd/N} + \|\bar \mu_{t_\gd,N} - \bar\pi_{N,\alpha}\|_{TV}
$$
where $\mu_{t_\gd,N} \in\cS_N$. Since $Ne^{-t_\gd/N}\to 0$, $N\to\infty$, from Proposition \ref{prop:muW}, \eqref{Eq:top1} follows if we show that 
\begin{equation}\label{Eq:top2}
\lim_{N\to\infty}N^{-\frac12}\mu_{t_\gd,N}(W) = 0. 
\end{equation}
By Proposition \ref{prop:Wt} we know that 
$$
0\leq \mu_{t_\gd}(W)  
 \leq 2KN^{\frac12(1-\gd)}\,.
$$ 
On the other hand, Lemma \ref{lem:mutn} shows that 
\begin{align*}
0\leq \mu_{t_\gd,N}(W) & =-\gamma_{t_\gd,N}^{-1}(1-\gamma_{t_\gd,N})\mu'_{t_\gd}(W) + \gamma_{t_\gd,N}^{-1}\mu_{t_\gd}(W)\\& 
 \leq 2\mu_{t_\gd}(W)=O\left(N^{\frac12(1-\gd)}\right)\,.
\end{align*}
\end{proof}
\subsection{Relaxation of the censored dynamics}\label{sec:censoreddyn}
Here we establish the second step in the proof of Proposition \ref{Prop:AbsCont}.
Consider the censored process obtained by suppressing all updates of the special particles. In other words, we use the censoring scheme $\cC$ such that $\cC(s)=\{u_1,\dots,u_{K-1}\}$, $s\geq 0$.
\begin{proposition}\label{prop:special2}
Fix $\ga> 0$. Let $P_{t,\cC}^x=\gd_x P_{t,\cC}$ and let $\pi_{N,\alpha}(\cdot|y)$ denote the equilibrium distribution given the special particle positions $y_i=x_{u_i}, i=1,\dots,K-1$. For any $\gd>0$ small enough,  setting $K=\lfloor \gd^{-1}\rfloor$, and $s_\gd = \gd\frac{\log N}{2\gap_N}$, one has, for $N$ sufficiently large
\begin{equation}\label{Eq:top10}
\sup_{x\in\gO_N}\| P_{s_\gd,\cC}^x- \pi_{N,\alpha}(\cdot|y)\|_{TV} \leq  \delta\;.
\end{equation}
\end{proposition}
\begin{proof}


By construction, the censored process corresponds to the product of $K$ independent adjacent walks each on the $n$-simplex, with $n:=\lfloor N/K\rfloor$. 
The bound \eqref{eq:roughUB} implies that when $N$ is sufficiently large, the mixing time of a system of size $n$ satisfies $T_n(\gep)\leq C n^2\log n$ for any given $\gep\in (0,1)$ if $n$ is large enough. Therefore,  
$$T_n( K^{-1} \delta)\le C \frac{N^2}{K^2}\log (N/K) \le  s_{\gd}\;,
$$
for $\delta > 0$ small enough. Thus if $d_N(t)$ is the distance defined in  
\eqref{defdn}, then 
$$
\sup_{x\in\gO_N}\| P_{s_\gd,\cC}^x- \pi_{N,\alpha}(\cdot|y)\|_{TV} \leq Kd_{n}(s_\gd)\le \delta\;,
$$
as required.
%
\end{proof}

\subsection{Proof of Proposition \ref{Prop:AbsCont}}\label{sec:proofAbsCont}
With the previous results at hand it is relatively simple to conclude the proof of the desired estimate. We formulate the result as follows. Recall the notation $ t_\delta=(1+\gd)\frac{\log N}{2\gap_N}$. 
\begin{proposition}\label{Prop:AbsCont*}
Fix $\ga\geq 1$. For any $\gd>0$, 
$$ \lim_{N\to\infty}\| P_{t_\gd}^\wedge - \pi_{N,\alpha} \|_{TV} =0\;.$$
\end{proposition}
\begin{proof}
Set $K=\lfloor \gd^{-1}\rfloor$ and let $\cC'$ denote the censoring scheme defined by
$\cC'(s)=\eset$ for $s\in[0, t_{\gd/2})$ and $\cC'(s)=\{u_1,\dots,u_{K-1}\}$ for $s\geq t_{\gd/2}$.
Let also $P^\wedge_{t,*}=\gd_\wedge P_{t,\cC'}$ denote the corresponding censored process. From Lemma \ref{cor:censor} we have
$$
\| P_{t_\gd}^\wedge - \pi_{N,\alpha} \|_{TV} \leq \| P_{t_\gd,*}^\wedge - \pi_{N,\alpha} \|_{TV}.
$$
A coupling of $P_{t_\gd,*}^\wedge$ and $\pi_{N,\alpha}$ can be achieved as follows. Call $\mu_t=\gd_\wedge P_t$. 
Use the optimal coupling attaining the total variation distance $\|\bar \mu_{t_{\gd/2}} - \bar\pi_{N,\alpha}\|_{TV}$
to couple the special particles at time $t_{\gd/2}$. If the special particles are coupled at time $t_{\gd/2}$
and if $x$ is the configuration at that time, then couple the remaining particles at time $t_\gd =t_{\gd/2}+s_{\gd/2}$ with the optimal coupling attaining the total variation distance $\| P_{s_{\gd/2},\cC}^x- \pi_{N,\alpha}(\cdot|y)\|_{TV}$, where $\cC$ is as in Proposition \ref{prop:special2}. This shows that
$$
\| P_{t_\gd}^\wedge - \pi_{N,\alpha} \|_{TV} \leq \|\bar \mu_{t_{\gd/2}} - \bar\pi_{N,\alpha}\|_{TV} + \sup_{x\in\gO_N}\| P_{s_{\gd/2},\cC}^x- \pi_{N,\alpha}(\cdot|y)\|_{TV}.
$$
From Proposition \ref{prop:special} and Proposition \ref{prop:special2}, 
$$
\limsup_{N\to\infty}\| P_{t_\gd}^\wedge - \pi_{N,\alpha} \|_{TV} \leq 2\gd.
$$
The distance $\| P_{t_\gd}^\wedge - \pi_{N,\alpha} \|_{TV}$ is decreasing as a function of $\gd$, and therefore we may take $\gd\to 0$ in the right hand side above to conclude. 
\end{proof}
\begin{remark}\label{rem:mono}
Proposition \ref{Prop:AbsCont*} can be strengthened to obtain that for any $\gd>0$, 
$$ \lim_{N\to\infty}\| \mu P_{t_\gd}- \pi_{N,\alpha} \|_{TV} =0\;.$$
for an arbitrary $\mu\in \cS$. Indeed, for any $\mu\in\cS$ one has $\mu P_t\in\cS$ and $\mu P_t\leq P^\wedge_t$  and therefore the claim follows from Lemma \ref{lem:kfs}.  However, at this point we cannot infer that the same holds for arbitrary initial distribution without the requirement that it belongs to $\cS$.
\end{remark}

\appendix

\section{An extension of the Randall-Winkler upper bound
}\label{App:roughUB}
As already discussed in the previous sections, one of the ingredients of our main results is the upper bound \eqref{Eq:roughUB} on the mixing time that captures the right order of magnitude up to a multiplicative constant, which in the case $\alpha=1$ was obtained by Randall and Winkler~\cite{RW05b}. In this section, we explain how one can extend that bound  to the Beta-resampling case with $\alpha \ge 1$. 
We are going to show more precisely that for any $\gep>0$, if  $N$ is sufficiently large, then  
\begin{equation}\label{eq:roughUB}
T_{N,\alpha}(\gep) \le 5 (\log N) \gap_N^{-1}.
\end{equation}

Let us fix $x\in \gO_N$ and construct a coupling of the two processes $\bX^x$ and 
$\bX^{\eq}$, 
where $\bX^{x}$ has initial state $x$ and 
$\bX^{\eq}$ has initial distribution $\pi_{N,\ga}$. Clearly, 
\begin{equation}
\| P^x_t-\pi_{N,\ga}\|_{TV}\le  \bbP[ \bX^x(t)\ne \bX^{\eq}(t)].
\end{equation}
Our manner of coupling the two processes is time dependent and is as follows: 
\begin{itemize}
                                                                               \item [(A)] For $t\le t_1:= 4 (\log N) \gap_N^{-1}$, we couple the two trajectories using the construction presented in  the proof of Proposition \ref{grandgradient}.
\item [(B)] For $t> t_1$, we couple the two trajectories according to the construction described in Section \ref{thecoupling}, namely we use a coupling that maximizes the probability of sticking at  each update.  
\end{itemize}
We first observe that at time $t_1$ the two processes have come close to one another: 
\begin{equation}\label{soclose}
 \bbE \left[\sum_{k=1}^{N-1}|X_k^x(t_1)- X^{\eq}_k(t_1)| \right] \le N^{-2}.
\end{equation}
To see this, referring to the proof of Proposition \ref{grandgradient}, we consider the two processes $(\bX^{\wedge}(t))_{t\in[0,t_1]}$,
$(\bX^{\vee}(t))_{t\in[0,t_1]}$ obtained by using the same clock processes $\cT$ and update variables $U$ as in the construction of 
$\bX^x(t)$ and  $\bX^{\eq}(t)$. Since this construction is order preserving, we necessarily have 
$$|X_k^x(t_1)- X^{\eq}_k(t_1)|\le X_k^{\wedge}(t_1)-X^{\vee}_k(t_1),$$
and \eqref{soclose} can be proved for $X_k^{\wedge}(t_1)-X^{\vee}_k(t_1)$ repeating the computation leading to \eqref{usingCS}.

\medskip

Now for $t>t_1$, we remark that in contrast with Randall-Winkler's case $\alpha=1$, the step B
described above does not descend from a monotone grand coupling, and thus it is not sufficient to couple $X^{\wedge}_k(t)$ and $X^{\vee}_k(t)$ by time $5\log N\gap_N^{-1}$ to prove the desired mixing time bound. 

Instead, we need to couple $\bX^x(t)$ and  $\bX^{\eq}(t)$ for arbitrary $x$. 
To this end we are going to repeat the recursive argument of Section \ref{sec:conclude} with $X^{\wedge}_k(t)$ replaced by $X^{\eq}_k(t)$. In particular, when we perform the update at a site $k$, 
we have to replace the interval $I^{\wedge}$ (see Section \ref{thecoupling}) by $I^{\eq}$, which we define as  
the resampling interval of $X_k^{\eq}$.
As a consequence of this modification, the supports of the densities $\rho_1$ and $\rho_3$ presented in \eqref{lesrhos} are not necessarily  intervals.  
Define the events 
\begin{gather*} \bar\cA^{(N)}_{1} := \left\{ \sum_{k=1}^{N-1}|X_k^x(t_1)- X^{\eq}_k(t_1)| < N^{-1} \right\},\\ 
\bar \cA^{(N)}_{2}:= \left\{ \min_{k\in \lint 1,N-1\rint, s\in [0, 2\log N]}  \nabla X^{\eq}_k(t_1+s) > \frac{1}{\sqrt{N} \log N} \right\}.
\end{gather*}
From \eqref{soclose} and Markov's inequality it follows that $\bbP(\bar\cA^{(N)}_{1})\to 1$, and from the proof of Lemma \ref{Lemma:BruteForce} it follows that $\bbP(\bar\cA^{(N)}_{2})\to 1$.

As in Section \ref{sec:conclude} we consider the probability of an unsuccessful update in the time interval $[t_1,t_1+2\log N]$. Let $\bar\cA^{(N)}_{3}$ denote the event that each site is updated at least once and no more than $10\log N$ times in this time interval. The probability of $\bar\cA^{(N)}_{3}$ tends to 1 (cf. Section \ref{sec:conclude}). 

Using the upper bound \eqref{I1I2} and recalling Remark \ref{extension}, we see that conditioned on the realization of the update times $(t_i^{(k)})_{i,k}$ in the time interval $[t_1,t_1+2\log N]$, on the event $\bar\cA^{(N)}_{1}\cap \bar\cA^{(N)}_{3}$, the probability that there exists at least one unsuccessful update within $[t_1,t_1+2\log N]$ is bounded above by $C\,N^{-1/2}(\log N)^2+o(1)$. This shows that the probability that $\bX^x(t_1+2\log N)\ne \bX^{\eq}(t_1+2\log N)$ vanishes as $N\to\infty$, which implies \eqref{eq:roughUB}. Note that the constant $5$ in this bound is not optimal. 

Finally, we remark that the proof given above can be extended to the case $\alpha\in(0,1)$, with $5$ replaced by a constant which depends on $\alpha$. However, this requires to prove an inequality that replaces the upper bound of \eqref{I1I2}, namely, that when $\alpha\in (0,1)$ there exists $C_\alpha>0$ such that for any pair of intervals
\begin{equation}\label{alphale1}
 \| \Beta_{\alpha}[l_1,r_1]- \Beta_{\alpha}[l_2,r_2] \|_{TV}\le  C_{\alpha}\left(\frac{\max(|l_2-l_1|, |r_2-r_1|)}{\max(r_1-l_1,r_2-l_2)}\right)^{1/\alpha}.
\end{equation}
We leave the details of the adaptation as well as the proof of \eqref{alphale1} to the interested reader.

\subsection*{Acknowledgements}
H.L. is grateful to Werner Krauth for bringing to his knowledge the work of Randall and Winkler and for a stimulating discussion on the subject. This work was initiated during a stay of P.C. and C.L. at IMPA, and continued during a stay of P.C. and H.L. at Dauphine: we are grateful for hospitality and support provided by these two institutions. H.L.\ also acknowledges support from a productivity grant of CNPq and a JCNE grant from FAPERj. C.L.\ acknowledges support from the grant SINGULAR ANR-16-CE40-0020-01.

\bibliographystyle{Martin}
\bibliography{library}

\begin{thebibliography}{KMP82}
\expandafter\ifx\csname url\endcsname\relax
  \def\url#1{\texttt{#1}}\fi
\expandafter\ifx\csname urlprefix\endcsname\relax\def\urlprefix{URL }\fi
\expandafter\ifx\csname href\endcsname\relax
  \def\href#1#2{#2}\fi
\expandafter\ifx\csname burlalt\endcsname\relax
  \def\burlalt#1#2{\href{#2}{\texttt{#1}}}\fi

\bibitem[AD87]{AD87}
\textsc{D.~Aldous} and \textsc{P.~Diaconis}.
\newblock Strong uniform times and finite random walks.
\newblock \emph{Adv. in Appl. Math.} \textbf{8}, no.~1, (1987), 69--97.
\newblock
  \burlalt{doi:10.1016/0196-8858(87)90006-6}{http://dx.doi.org/10.1016/0196-8858(87)90006-6}.

\bibitem[AF02]{aldous2002reversible}
\textsc{D.~Aldous} and \textsc{J.~Fill}.
\newblock Reversible markov chains and random walks on graphs, 2002.
\newblock \urlprefix\url{https://www.stat.berkeley.edu/~aldous/RWG/book.pdf}.

\bibitem[Cap08]{Cap08}
\textsc{P.~Caputo}.
\newblock On the spectral gap of the kac walk and other binary collision
  processes.
\newblock \emph{Alea} \textbf{4}, (2008), 205--222.

\bibitem[CCL03]{CCL2003}
\textsc{E.~A. Carlen}, \textsc{M.~C. Carvalho}, and \textsc{M.~Loss}.
\newblock Determination of the spectral gap for kac's master equation and
  related stochastic evolution.
\newblock \emph{Acta mathematica} \textbf{191}, no.~1, (2003), 1--54.

\bibitem[DFK91]{dyer1991random}
\textsc{M.~Dyer}, \textsc{A.~Frieze}, and \textsc{R.~Kannan}.
\newblock A random polynomial-time algorithm for approximating the volume of
  convex bodies.
\newblock \emph{Journal of the ACM (JACM)} \textbf{38}, no.~1, (1991), 1--17.

\bibitem[Dia96]{diaconis1996cutoff}
\textsc{P.~Diaconis}.
\newblock The cutoff phenomenon in finite markov chains.
\newblock \emph{Proceedings of the National Academy of Sciences} \textbf{93},
  no.~4, (1996), 1659--1664.

\bibitem[Gia02]{Giacomin}
\textsc{G.~Giacomin}.
\newblock \emph{{Aspects of statistical mechanics of random surfaces}}.
\newblock Lecture Notes for course given at IHP. 2002.

\bibitem[GS02]{Gibbs02}
\textsc{A.~L. Gibbs} and \textsc{F.~E. Su}.
\newblock On choosing and bounding probability metrics.
\newblock \emph{INTERNAT. STATIST. REV.}  419--435.

\bibitem[HJ17]{HoughJiang17}
\textsc{B.~Hough} and \textsc{Y.~Jiang}.
\newblock Cut-off phenomenon in the uniform plane {K}ac walk.
\newblock \emph{Ann. Probab.} \textbf{45}, no.~4, (2017), 2248--2308.
\newblock
  \burlalt{doi:10.1214/16-AOP1111}{http://dx.doi.org/10.1214/16-AOP1111}.

\bibitem[HLP16]{HLP16}
\textsc{J.~Hermon}, \textsc{H.~Lacoin}, and \textsc{Y.~Peres}.
\newblock Total variation and separation cutoffs are not equivalent and neither
  one implies the other.
\newblock \emph{Electron. J. Probab.} \textbf{21}, (2016), Paper No. 44, 36.
\newblock
  \burlalt{doi:10.1214/16-EJP4687}{http://dx.doi.org/10.1214/16-EJP4687}.

\bibitem[KMP82]{KMP82}
\textsc{C.~Kipnis}, \textsc{C.~Marchioro}, and \textsc{E.~Presutti}.
\newblock Heat flow in an exactly solvable model.
\newblock \emph{J. Statist. Phys.} \textbf{27}, no.~1, (1982), 65--74.
\newblock
  \burlalt{doi:10.1007/BF01011740}{http://dx.doi.org/10.1007/BF01011740}.

\bibitem[Lac16]{Lac16}
\textsc{H.~Lacoin}.
\newblock Mixing time and cutoff for the adjacent transposition shuffle and the
  simple exclusion.
\newblock \emph{Ann. Probab.} \textbf{44}, no.~2, (2016), 1426--1487.
\newblock
  \burlalt{doi:10.1214/15-AOP1004}{http://dx.doi.org/10.1214/15-AOP1004}.

\bibitem[Lig05]{Liggettbook}
\textsc{T.~M. Liggett}.
\newblock \emph{Interacting particle systems}.
\newblock Classics in Mathematics. Springer-Verlag, Berlin, 2005.
\newblock Reprint of the 1985 original.

\bibitem[LL18]{LabLacWASEP}
\textsc{C.~{Labb{\'e}}} and \textsc{H.~{Lacoin}}.
\newblock {Mixing time and cutoff for the weakly asymmetric simple exclusion
  process}.
\newblock \emph{ArXiv e-prints} (2018).
\newblock \burlalt{arXiv:1805.12213}{http://arxiv.org/abs/1805.12213}.

\bibitem[LPW17]{LevPerWil}
\textsc{D.~A. Levin}, \textsc{Y.~Peres}, and \textsc{E.~L. Wilmer}.
\newblock \emph{Markov chains and mixing times}.
\newblock American Mathematical Society, Providence, RI, 2017.
\newblock Second edition of [ MR2466937], With a chapter on ``Coupling from the
  past'' by James G. Propp and David B. Wilson.

\bibitem[Luk55]{lukacs1955}
\textsc{E.~Lukacs}.
\newblock A characterization of the gamma distribution.
\newblock \emph{The Annals of Mathematical Statistics} \textbf{26}, no.~2,
  (1955), 319--324.

\bibitem[LV03]{lovasz2003hit}
\textsc{L.~Lov{\'a}sz} and \textsc{S.~Vempala}.
\newblock Hit-and-run is fast and fun.
\newblock \emph{preprint, Microsoft Research} (2003).

\bibitem[Pet75]{Petrov}
\textsc{V.~V. Petrov}.
\newblock \emph{Sums of independent random variables}.
\newblock Springer-Verlag, New York-Heidelberg, 1975.
\newblock Translated from the Russian by A. A. Brown, Ergebnisse der Mathematik
  und ihrer Grenzgebiete, Band 82.

\bibitem[Pre74]{Preston}
\textsc{C.~J. Preston}.
\newblock A generalization of the {${\rm FKG}$} inequalities.
\newblock \emph{Comm. Math. Phys.} \textbf{36}, (1974), 233--241.

\bibitem[PW13]{PWcensoring}
\textsc{Y.~Peres} and \textsc{P.~Winkler}.
\newblock Can extra updates delay mixing?
\newblock \emph{Communications in Mathematical Physics} \textbf{323}, no.~3,
  (2013), 1007--1016.

\bibitem[RW05a]{RW05b}
\textsc{D.~Randall} and \textsc{P.~Winkler}.
\newblock Mixing points on a circle.
\newblock In \emph{Approximation, Randomization and Combinatorial Optimization.
  Algorithms and Techniques},  426--435. Springer, 2005.

\bibitem[RW05b]{RW05}
\textsc{D.~Randall} and \textsc{P.~Winkler}.
\newblock Mixing points on an interval.
\newblock In \emph{Proceedings of the Second Workshop on Analytic Algorithms
  and Combinatorics, Vancouver, 2005},  216--221. 2005.

\bibitem[Smi13]{Smith2}
\textsc{A.~Smith}.
\newblock Analysis of convergence rates of some gibbs samplers on continuous
  state spaces.
\newblock \emph{Stochastic Processes and their Applications} \textbf{123},
  no.~10, (2013), 3861--3876.

\bibitem[Smi14]{Smith1}
\textsc{A.~Smith}.
\newblock A gibbs sampler on the $ n $-simplex.
\newblock \emph{The Annals of Applied Probability} \textbf{24}, no.~1, (2014),
  114--130.

\bibitem[Wil04]{Wil04}
\textsc{D.~B. Wilson}.
\newblock Mixing times of {L}ozenge tiling and card shuffling {M}arkov chains.
\newblock \emph{Ann. Appl. Probab.} \textbf{14}, no.~1, (2004), 274--325.
\newblock
  \burlalt{doi:10.1214/aoap/1075828054}{http://dx.doi.org/10.1214/aoap/1075828054}.

\end{thebibliography}

\end{document}